\documentclass[12pt]{amsart}
\usepackage{amsmath,amscd,graphics,color,a4wide,verbatim}
\usepackage{tgpagella}
\usepackage{stmaryrd}

\usepackage{amssymb,amsthm,mathrsfs,enumitem}
\usepackage[parfill]{parskip}
\usepackage[dvipsnames]{xcolor}
\usepackage{tikz, tikz-3dplot, pgfplots}
\usepackage{tkz-graph,tikz-cd}
\usetikzlibrary{matrix,arrows,decorations.pathmorphing,positioning,patterns}
\usepackage{subcaption,ytableau,dynkin-diagrams,slashed,textcmds,epic,marvosym}
\usepackage{multicol,fullpage}

\usepackage[colorlinks]{hyperref}

\newtheorem{theorem}{Theorem}[section]
\newtheorem{lemma}[theorem]{Lemma}
\newtheorem{lem-def}[theorem]{Lemma-definition}
\newtheorem{proposition}[theorem]{Proposition}
\newtheorem{prop-def}[theorem]{Proposition-definition}

\newtheorem{corollary}[theorem]{Corollary}

\newtheorem*{setup}{Set-up}

\theoremstyle{definition}
\newtheorem{definition}[theorem]{Definition}

\newtheorem{examples}[theorem]{Examples}

\theoremstyle{remark}
\newtheorem{remark}[theorem]{Remark}
\numberwithin{equation}{section}

\pgfplotsset{compat=1.18}

\newcommand{\CC}{\mathbb{C}}

\newcommand{\GG}{\mathbb{G}}

\newcommand{\LL}{\mathbb{L}}

\newcommand{\PP}{\mathbb{P}}

\newcommand{\TT}{\mathbb{T}}

\newcommand{\shH}{\mathcal{H}}

\newcommand{\shO}{\mathcal{O}}

\newcommand{\codim}{\operatorname{codim}}

\renewcommand{\ker}{\operatorname{Ker}} % Only define once

\newcommand{\Ext}{\operatorname{Ext}}
\newcommand{\sExt}{\mathscr{E}\kern -3pt xt}

\newcommand{\Hom}{\operatorname{Hom}}
\newcommand{\sHom}{\mathscr{H}\kern-5pt\mathcalligra{om}}

\newcommand{\Pic}{\operatorname{Pic}}

\newcommand{\Spec}{\operatorname{Spec}}

\newcommand{\Sym}{\operatorname{Sym}}

\newcommand{\Tot}{\operatorname{Tot}}

\newcommand{\arw}{\longrightarrow}

\newcommand{\Hdot}{\operatorname{H}}

\DeclareRobustCommand{\Sec}{\ifmmode\mathsection\else\textsection\fi}

\title[]{Lagrangian Intersections, Symplectic Reduction and Kirwan Surjectivity}
\author[NC. Leung]{Naichung Conan Leung}
\address{The Institute of Mathematical Sciences and Department of Mathematics\\The Chinese University of Hong Kong\\Shatin, N.T.\\ Hong Kong}
\email{leung@ims.cuhk.edu.hk}

\author[Y. Xie]{Ying Xie}
\address{School of Mathematical Sciences\\
Shenzhen University\\ Shenzhen, 518061, P.R. China}
\email{xieying@szu.edu.cn}

\author[YT. Yau]{Yu Tung Yau}
\address{Kavli Institute for the Physics and Mathematics of the Universe (WPI)\\ The University of Tokyo Institutes for Advanced Study\\ The University of Tokyo\\
Kashiwa, Chiba 277-8583, Japan}
\email{yu-tung.yau@ipmu.jp}

\begin{document}

\begin{abstract}
Given a smooth holomorphic symplectic variety $X$ with a Hamiltonian $G$-action, $G$-invariant Lagrangians $C's$ induce Lagrangians in the symplectic quotient $X\sslash  G$. Given clean intersections $B=C_1\cap C_2$ whose conormal sequence splits, we show that 
$$C_1/G\times_{X\sslash G} C_2/G\cong T^{\vee}[-1](B/G).$$
When $\det(N_{B/C_2})$ is torsion, we have $\Ext^{\bullet}_{X\sslash G}(\shO_{C_1/G}, \shO_{C_2/G})\cong \Hdot^{\bullet}_G(B, \det(N_{B/C_2})_{\delta})$ provided that the Hodge-to-de Rham degeneracy holds. Furthermore, we have a generalized version of Kirwan surjectivity
$\Ext^{\bullet}_{X\sslash G}(\shO_{C_1/G}, \shO_{C_2/G})\twoheadrightarrow \Ext^{\bullet}_{X^{ss}\sslash G}(\shO_{C_1^{ss}/G}, \shO_{C_2^{ss}/G})$ if $B$ is proper. When $C_1=C_2$, this is the Kirwan surjectivity, which is now interpreted as the symmetry commutes with reduction problem in 3d B-model. We also obtain similar results for $K_{C_1/G}^{1/2}$ and $K_{C_2/G}^{1/2}$. 
\end{abstract}

\maketitle
\section{Introduction}
As is central in both symplectic geometry and mirror symmetry, the study of complex Lagrangian intersections is key to constructing the Fukaya category of complex Lagrangians and 3d B-model 2-category proposed by Kapustin-Rozansky-Saulina \cite{kapustin2009three}. In this paper, we investigate an equivariant Lagrangian intersection. 
\begin{setup}[S]\label{setup}
$(X,\omega)$ is a smooth symplectic quasi-projective variety with a Hamiltonian $G$-action, where $G$ is a complex linear reductive group, and $C_1, C_2$ are $G$-invariant smooth connected Lagrangians of $(X, \omega)$ such that the scheme theoretic intersection $B:=C_1\cap C_2$ is non-empty and smooth.
\end{setup}
We call conormal sequence of $G$-invariant $B\subset X$ \emph{splits} if the following conormal bundle sequence splits as $G$-equivariant bundles
\begin{align*}
    0\arw N^{\vee}_{B/X}\arw \Omega_X|_B \arw \Omega_B\arw 0. 
 \end{align*}
For example, given any smooth quasi-projective variety $M$ with a group action $G$, the cotangent bundle $X=\Tot(\Omega_M)$ is a smooth symplectic quasi-projective variety with a Hamiltonian $G$-action. The zero section $M\subset X$ is always a $G$-invariant Lagrangian of $X$, and the conormal sequence of $M\subset X$ splits canonically.

In Set-up (S), there is a complex moment map $\mu: X\rightarrow \mathfrak{g}^{\vee}$, and $G$-invariant Lagrangians $C_1$ and $C_2$ share the same moment value $c$ (see Lemma \ref{lem: invLag}). The symplectic quotient $X\sslash_c G$ of level $c$ ($X\sslash G$ for short) is a 0-shifted symplectic Artin stack, and the quotient stacks $C_1/G$ and $C_2/G$ are (0-shifted) Lagrangians of $X\sslash G$.
Let $\LL_{B/G}$ denote the cotangent complex of the quotient stack $B/G$, and denote the $(-1)$-shifted cotangent stack of $B/G$ as $\Tot_{B/G}(\LL_{B/G}[-1]):=\Spec_{B/G}(\Sym^{\bullet} \LL_{B/G}^{\vee}[1])$. If $\mathfrak{L}\in \Pic(B)$ is a torsion line bundle on $B$, \emph{i.e.,} $\mathfrak{L}^{\otimes n}\cong \shO_B$ for some integer $n$, then $\mathfrak{L}$ determines a rank 1 $\CC$-local system in the usual analytic topology on $B$, denoted as $\mathfrak{L}_{\delta}$. Similarly for equivariant cases. Our primary result establishes the geometry of Lagrangian intersections in $X\sslash  G$.  

\begin{theorem}[=Theorem \ref{mainthm}+Corollary \ref{symredext2}]
In Set-up (S), if the conormal sequence of $B\subset X$ splits, then the derived fiber product 
    $$C_1/G\times_{X\sslash  G} C_2/G\cong \Tot_{B/G}(\LL_{B/G}[-1])$$
If $\det(N_{B/C_2})$ is torsion in $\Pic(B/G)$, $B$ is proper-over-affine and $h^0(B, \shO)^{G}<\infty$, then 
    $$\Ext^{\bullet}_{X\sslash G}(\shO_{C_1/G}, \shO_{C_2/G})\cong \Hdot^{\bullet-\codim(B, C_2)}_G(B, \det(N_{B/C_2})_{\delta}).$$
   In particular, if $C_1=C_2=C$ is smooth and proper, then $$\Ext^{\bullet}_{X\sslash G}(\shO_{C/G}, \shO_{C/G})\cong \Hdot^{\bullet}_G(C, \CC).$$
   \end{theorem}
The local system appears naturally in examples of Lagrangian intersections. For instance, given a $G$-invariant Morse-Bott regular function  $f: M\arw \CC$ on a smooth variety $M$, such that the critical subscheme $Z:=Z(df)\subset X$ is smooth. Then $N_{Z/M}^{\vee}\cong N_{Z/M}$ by the nondegeneracy of the Hessian $\nabla^2 f\in \Hdot^0(Z, \Sym^2 N_{Z/M}^{\vee})$. Thus 
$\det(N_{Z/M})^{\otimes 2}\cong \shO_Z$. Then $C_1=M$, the zero section and $C_2=\Gamma_{df}$, the graph of $df$ are $G$-invariant Lagrangians of $X$ with a smooth intersection $Z$, and the conormal sequence $Z\subset X$ splits. Therefore,
$$\Ext_{X\sslash G}^{\bullet}(\shO_{C_1/G}, \shO_{C_2/G})\cong \Hdot^{\bullet-\codim(Z, M)}(Z/G, \det(N_{Z/M})_{\delta}).$$
Another example of a 2-torsion local system arises when the canonical bundle $K_{C_i}$ of Lagrangians $C_i$ admit square roots $K_{C_i}^{1/2}$ which is equivalent to the existence of spin structures on $C_i~ (i=1,2)$. Then $K_{C_1}^{\vee}|_B\otimes K_{C_2}|_B\otimes \det(N_{B/C_2})^{\otimes 2}\cong \shO_B$ (see Lemma \ref{lem: canlag}) and $(K_{C_1}^{-1/2}|_B\otimes K_{C_2}^{1/2}|_B\otimes \det(N_{B/C_2}))_{\delta}$ is a 2-torsion $\CC$-local system. Physically, this is significant when these Lagrangians are interpreted as branes in a superstring theory, where each brane must carry a spin (or $\text{Spin}^c$) structure to ensure that worldsheet fermions are well-defined (see \cite{freed1999anomalies}). 

\begin{theorem}[=Theorem \ref{thm: can}]
    In Set-up (S), if the conormal sequence of $B\subset X$ splits, and $K_{C_i/G }^{1/2}~(i=1,2)$ exist. Suppose that $B$ is proper-over-affine and $h^0(B, \shO)^{G}<\infty$, then
    $$\Ext^{\bullet}_{X\sslash G}(K^{1/2}_{C_1/G}, K^{1/2}_{C_2/G})\cong \Hdot^{\bullet-\codim(B, C_2)}_G(B, (K_{C_1/G}^{-1/2}|_{B/G}\otimes K_{C_2/G}^{1/2}|_{B/G}\otimes \det(N_{B/C_2}))_{\delta}).$$
\end{theorem}

For $X$ with a $G$-equivariant ample bundle $L$, the semistable reduction $X^{ss}(L)\sslash  G$ ($X^{ss}\sslash  G$ for short) and Lagrangian $C^{ss}(L|_C)/G$ ($C^{ss}/G$ for short) yield a functorial restriction map
\begin{align}\label{map: ext}
    \Ext_{X\sslash  G}^{\bullet}(\shO_{C_1/G}, \shO_{C_2/G})\arw \Ext_{X^{ss}\sslash  G}^{\bullet}(\shO_{C_1^{ss}/G}, \shO_{C_2^{ss}/G}),
\end{align}
which is equivalent to a twisted Kirwan map if $\det(N_{B/C_2})$ is torsion:
\begin{align}\label{map: kirwan}
    \Hdot^{\bullet}_G(B, \det(N_{B/C_2})_{\delta})\arw \Hdot^{\bullet}_G(B^{ss}, 
\det(N_{B/C_2})_{\delta}).
\end{align} 
The twisted Kirwan map (\ref{map: kirwan}) and also (\ref{map: ext}) are surjective by a twisted Kirwan theorem  \ref{thm:twistedkirwan} in the Appendix.
\begin{theorem}[=Theorem \ref{mainthm2}+Theorem \ref{thm: kircan}] In Set-up (S), assume that the conormal sequence $B\subset X$ splits. Given a $G$-equivariant ample bundle $L$ on $X$. Suppose that  $B$ is proper over affine, and $h^0(B, \shO)^G<\infty$.
    \begin{enumerate}
        \item If $\det(N_{B/C_2})$ is torsion in $\Pic(B/G)$, then the natural restriction 
    $$\Ext_{X\sslash  G}^{\bullet}(\shO_{C_1/G}, \shO_{C_2/G})\arw \Ext_{X^{ss}\sslash  G}^{\bullet}(\shO_{C_1^{ss}/G}, \shO_{C_2^{ss}/G})$$
    is surjective.
    \item If $K_{C_1/G}^{1/2}$ and $K_{C_2/G}^{1/2}$ exist, then the natural restriction 
    $$\Ext_{X\sslash  G}^{\bullet}(K_{C_1/G}^{1/2}, K_{C_2/G}^{1/2})\arw \Ext_{X^{ss}\sslash  G}^{\bullet}(K_{C^{ss}_1/G}^{1/2}, K_{C^{ss}_2/G}^{1/2})$$
    is surjective.
    \end{enumerate}
\end{theorem}

\subsection*{Outline} Section \ref{sec: clean} reviews clean intersections, Ext groups, and excess bundles. Section \ref{sec: lagint} analyzes Lagrangian intersections and excess sequences. Section \ref{sec:eqsym} proves our main Theorem \ref{mainthm} for derived Lagrangian intersections of $X\sslash  G$ and its several consequences. Section \ref{sec:git} discusses the relationship between the Ext groups of equivariant Lagrangians and their semi-stable quotient via the twisted Kirwan surjectivity. The Appendix discusses the twisted Kirwan surjectivity and gives a detailed proof of it based on Kirwan's original argument. 

\subsection*{Related works} In \cite{gunningham2023deformation},  Gunningham and Safronov describe the RHom complex between a pair of complex Lagrangians of $X$ in terms of deformation quantization. In \cite{mladenov2024formality}, Mladenov relates the RHom complex of a pair of complex Lagrangians to the de Rham cohomology of the intersection with the coefficient in a 2-torsion local system under the assumption of smoothness of the intersection (i.e., clean intersection). They are related to the Hochschild homology of the 3d B-model 2-category by Kapustin-Rozansky-Saulina \cite{kapustin2009three}. 

\subsection*{Acknowledgments} 
This research is substantially supported by grants from the Research Grants Council of the Hong Kong Special Administrative Region, China (Project No. CUHK14306322, CUHK14305923 and CUHK14302224) and direct grants from the Chinese University of Hong Kong. The first and second authors thank Yalong Cao for insightful discussions on shifted symplectic geometry. The second author is particularly grateful to Qingyuan Jiang for both discussions on derived algebraic geometry and the hospitality during his visit at the Hong Kong University of Science and Technology. Special thanks also go to Ki-Fung Chan and Leon Li for sharing their expertise on mirror symmetry and real Lagrangian Floer theory. The authors extend their appreciation to Ziming Ma and Michael McBrean for valuable advice on this project.

\subsection*{Notations} Throughout this paper, $G$ is a linear reductive group over $\CC$ and all varieties, stacks and derived stacks are over $\CC$. For a $G$-variety $X$, we write the quotient stack of $X$ by $G$ as $X/G$. For any two morphisms between derived stacks with the same target $C_1\arw X$ and $C_2\arw X$, we write $C_1\times_X C_2$ as the derived fiber product of the two morphisms. Functors $\otimes, f_*, f^{*}$ are viewed as derived functors, and $\shH^i(-\otimes-), \shH^i(f_*), \shH^i(f^{*})$ denote their cohomology. Denote $D_{qc}(Y)$ the $\infty$-category of quasi-coherent sheaves on a derived stack $Y$. For any vector bundle or perfect complex $E$ on $Y$, total space of $E$ means $\Spec_Y{\Sym^{\bullet}(E^{\vee})}$, which is denoted as $\Tot_Y(E)$.

\section{Clean intersection and Excess bundles}\label{sec: clean}
Let $X$ be a smooth quasi-projective variety with two smooth closed subvarieties $j_1: C_1\hookrightarrow X$, $j_2: C_2\hookrightarrow X$. Consider the intersection diagram \begin{equation}\label{intdiag} 
\begin{tikzcd}
B=C_1\cap C_2 \arrow[r,hook, "i_2"] \arrow[d,hook,"i_1"]
& C_2 \arrow[d,hook', "j_2"'] \\
 C_1 \arrow[r,hook, "j_1"]
& X.
\end{tikzcd}
 \end{equation}
\subsection{Basic Definition and Transversality}
\begin{definition}
The intersection of $C_1$ and $C_2$ is called \emph{clean} if the scheme-theoretic intersection $B:=C_1\cap C_2$ is smooth. 
The quotient bundle $E=\dfrac{TX|_B}{TC_1|_B+TC_2|_B}$ is called the \emph{excess bundle} of the intersection of $C_1$ and $C_2$.
\end{definition}

\begin{remark}
\begin{enumerate}
    \item Transverse intersections occur when $T_xC_1+T_xC_2=T_xX$ holds for every $x\in B$, equivalently $E=0$. Thus $E$ can be viewed as a measure of the failure of transversality of the intersection.
    \item When $C_1=C_2=B$, the excess bundle $E$ equals to the normal bundle $N_{B/X}$.
\end{enumerate}
 \end{remark}
 The next lemma gives an alternative description of the excess bundle.
\begin{lemma}\label{excess}
   There is a short exact sequence of vector bundles on $B$:
$$0\arw N_{B/C_2}\arw N_{C_1/X}|_B\arw E\arw 0.$$
\end{lemma}

\begin{proof}
    The kernel of the surjective morphism $N_{C_1/X}|_B=\dfrac{TX|_B}{TC_1|_B}\arw  E=\dfrac{TX|_B}{TC_1|_B+TC_2|_B}$ is $\dfrac{TC_1|_B+TC_2|_B}{TC_1|_B}\cong\dfrac{TC_2|_B}{TB}=N_{B/C_2}$.    
\end{proof}
\begin{remark}
    The description of the excess bundle in Lemma \ref{excess} holds for derived schemes or stacks by replacing the normal bundles with the 1-shifted relative tangent complexes. See subsection \ref{subsec: excess}.
\end{remark}

\subsection{Derived Intersections and Ext groups}\label{subsec: derint}
We aim to understand 
\begin{align}
    \Ext_X^{\bullet}(j_{1*}\shO_{C_1}, j_{2*}\shO_{C_2}),
    \end{align}
and more generally, 
\begin{align}
    \Ext_X^{\bullet}(j_{1*}F_1, j_{2*}F_2),
    \end{align}
where $F_1$ and $F_2$ are perfect complexes on $C_1$ and $C_2$ respectively. 
By the adjunction of pullback-pushforward,      \begin{align}\label{ext}
    & \Ext_X^{\bullet}(j_{1*}\shO_{C_1}, j_{2*}\shO_{C_2})\cong \Ext_{C_2}^{\bullet}(j_2^{*}j_{1*}\shO_{C_1}, \shO_{C_2}),\\ 
    & \Ext_X^{\bullet}(j_{1*}F_1, j_{2*}F_2)\cong \Ext_{C_2}^{\bullet}(j_2^*j_{1*}F_1, F_2). \label{ext'}
\end{align}

The base-change property of the intersection diagram (\ref{intdiag}) gives a natural morphism
$$ j_2^{*}j_{1*}\shO_{C_1}\arw i_{2*}i_1^{*}\shO_{C_1}=i_{2*}\shO_B,$$
which is not an isomorphism unless $E$ is trivial. 
To address this issue, we introduce the following derived intersection diagram
\begin{equation}\label{derintdiag}
   \begin{tikzcd}
    B
    \arrow[drr, bend left,"i_1"] \arrow[ddr, bend right, "i_2"] \arrow[dr, "\gamma"] &&\\
    & C_1\times_X C_2\arrow[r,"\tilde{j_2}"] \arrow[d,"\tilde{j_1}"]\arrow[ul, dotted, shift left=2, "\pi"]
& C_1 \arrow[d, "j_1"] \\
 & C_2 \arrow[r, "j_2"]
& X.
   \end{tikzcd}
  \end{equation}
The morphism $\gamma$ is provided by the universal property of the derived fiber product $C_1\times_X C_2$, which is an isomorphism on the classical truncation: $\gamma_{cl}: B\cong (C_1\times_X C_2)_{cl}$. In other words,
$$\shH^0(\shO_{C_1}\otimes_{\shO_X} \shO_{C_2})\cong \shO_B.$$
The derived scheme $C_1\times_X C_2$ can be viewed as a thickening of its classical truncation $B$. Roughly speaking, a derived scheme over $\CC$ is a derived ringed space $(Y, \shO_Y)$, where $\shO_Y$ is a sheaf of commutative differential graded algebras over $\CC$ with quasi-coherent cohomology sheaves on its classical truncation scheme $(Y, \shH^0(\shO_Y))$. The readers can refer to \cite{lurie2004derived} and \cite{toen2014derived} for background and more details on derived algebraic geometry.

By base-change property,
$$j_2^{*}j_{1*}\shO_{C_1}\cong \tilde{j_2}_*\shO_{C_1\times_X C_2},\qquad j_2^{*}j_{1*}F_1\cong \tilde{j_2}_*\tilde{j_1}^*F_1.$$ If the inclusion $\gamma: B\arw C_1\times_X C_2$ \emph{splits} as derived schemes over $C_1\times C_2$, \emph{i.e.}, $\gamma$ admits  a left inverse $\pi$ (which is dotted in (\ref{derintdiag})) such that both $\tilde{j_1}$ and $\tilde{j_2}$ factors $\pi$, then
\begin{align}\label{eq:bc1}
& \tilde{j_2}_*\shO_{C_1\times_X C_2}=i_{1*}\pi_*\shO_{C_1\times_X C_2},\\
&\tilde{j_2}_*\tilde{j_1}^*F_1=i_{1*}\pi_*\pi^*i_2^*F_1\cong i_{1*}(\pi_*\shO_{C_1\times_X C_2}\otimes i_2^*F_1),\label{eq:bc2}
\end{align}
where the last $\cong$ comes from the projection formula. 
Thus, the calculation of the Ext groups (\ref{ext}) and (\ref{ext'}) can be translated to calculate $\pi_*\shO_{C_1\times_X C_2}$ if such a left inverse exists. For that, we need a stronger condition that 
\begin{align}\label{fib}
    C_1\times_X C_2\cong \Tot_B(F),
\end{align}
where $F$ is a perfect complex on $B$, and (\ref{fib}) is called a \emph{fibration} structure on the derived intersection in the sense of \cite{arinkin2012self}. In loc. cit, the authors study the self-intersection case, and later clean intersection case is considered in \cite{arinkin2019formality} by Arikin-C\u{a}ld\u{a}raru-Hablicsek (or see Theorem \ref{ACH} below). For instance, if $X=M\times M$, and $C_1=C_2=\Delta_M\subset X$ the diagonal, then $\shO_{\Delta}\otimes_{M\times M}\shO_{\Delta}=\Spec_M(\Sym(\Omega_M[1]))$ which is known as the HKR isomorphism \cite{swan1996hochschild}. 

\subsection{Cohomology Sheaves of Derived Tensor Product}
\begin{proposition}[C\u{a}ld\u{a}raru-Katz-Sharpe\cite{caldararu2003d}]\label{tor}
    Let $C_1$ and $C_2$ be two smooth subvarieties of $X$ with a clean intersection $B$, and $E$ be the excess bundle. Then for each $k$,
$$\shH^{-k}(\shO_{C_1}\otimes_{\shO_X} \shO_{C_2})\cong \wedge^k E^{\vee}.$$
\end{proposition}
By the Grothendieck spectral sequence, it is easy to get the following proposition stated in \cite{caldararu2003d}.
\begin{proposition}[C\u{a}ld\u{a}raru-Katz-Sharpe\cite{caldararu2003d}]\label{prop: ss}
    Let $C_1$ and $C_2$ be two smooth subvarieties of $X$ with a clean intersection $B$, and $E$ be the excess bundle. Then there are spectral sequences
    \begin{align}
        &E_2^{p,q}=\Hdot^q(B, \wedge^{p-m}E\otimes det(N_{B/C_2}))\Longrightarrow \Ext_X^{p+q}(j_{1*}\shO_{C_1},j_{2*}\shO_{C_2}),\\
        &E_2^{p,q}=\Hdot^q(B, (F_1^{\vee}\otimes F_2)|_B\otimes\wedge^{p-m}E\otimes \det(N_{B/C_2}))\Longrightarrow \Ext_X^{p+q}(j_{1*}F_1,j_{2*}F_2),
    \end{align}
where $m=\codim (B, C_2)$, and $F_1$(resp. $F_2$) is a perfect complex on $C_1$(resp. $C_2$).
\end{proposition}

\subsection{Criterion on Fibration Structure}
Proposition \ref{tor} implies that $C_1\times_X C_2$ looks like $\Tot_B(E[-1])$ locally. However, it is usually not the case globally. Here is a criterion for $C_1\times_X C_2$ being isomorphic to $\Tot_B(E[-1])$ obtained by Arikin-C\u{a}ld\u{a}raru-Hablicsek\cite{arinkin2019formality}. 
\begin{theorem}[Arikin-C\u{a}ld\u{a}raru-Hablicsek \cite{arinkin2019formality}]\label{ACH}
Let $C_1$ and $C_2$ be two smooth subvarieties of $X$ with a clean intersection $B$, and $E$ be the excess bundle. Then the derived fiber product
 $C_1\times_X C_2\cong \Tot_B(E[-1])=\Spec_B(\Sym^{\bullet} E^{\vee}[1])$ as derived schemes over $C_1\times C_2$ if and only if the excess bundle sequence
 \begin{align}\label{excessplit}
    0\arw TC_1|_B+TC_2|_B\arw TX|_B \arw E\arw 0 
 \end{align}
splits.
\end{theorem}

If $C_1\times_X C_2\cong \Tot_B(E[-1])$, then there is a left inverse $\pi: C_1\times_X C_2\arw B$ to $\gamma: B\arw C_1\times_X C_2$ induced from the bundle structure. Then (\ref{eq:bc1}) and (\ref{eq:bc2})
\begin{align*}
   & j_2^{*}j_{1*}\shO_{C_1}\cong i_{1*}\pi_*(\shO_{\Tot_B(E[-1])})=i_{1*}(\Sym^{\bullet} (E^{\vee}[1]))=\bigoplus_{k\geq 0}i_{1*}\wedge^k E^{\vee},\\ 
   & j_2^{*}j_{1*}F_1\cong i_{1*}(\pi_*\shO_{\Tot_B(E[-1])}\otimes i_2^*F_1)=i_{1*}(\Sym^{\bullet} (E^{\vee}[1])\otimes i_2^*F_1)=\bigoplus_{k\geq 0}i_{1*}(\wedge^k E^{\vee}\otimes i_2^*F_1).
\end{align*}
Hence, the spectral sequences of Proposition \ref{prop: ss} degenerate at the $E_2$ page. 
\begin{corollary}\label{ssd}
Let $C_1$ and $C_2$ be two smooth subvarieties of $X$ with a clean intersection $B$, and $E$ be the excess bundle. Assume further that the excess sequence (\ref{excessplit}) splits, then 
\begin{align*}
   & \Ext_X^{\bullet}(j_{1*}\shO_{C_1},j_{2*}\shO_{C_2})\cong\bigoplus_{p+q+m=\bullet}
\Hdot^q(B, \wedge^{p-m}E\otimes \det(N_{B/C_2})),\\
& \Ext_X^{\bullet}(j_{1*}F_1,j_{2*}F_2)\cong\bigoplus_{p+q+m=\bullet}
\Hdot^q(B, (F_1^{\vee}\otimes F_2)|_B\otimes \wedge^{p-m}E\otimes \det(N_{B/C_2}))
\end{align*}
where $m=\codim (B, C_2)$, and $F_1$(resp. $F_2$) is a perfect complex on $C_1$(resp. $C_2$).
\end{corollary}

In particular, for the case $C_1=C_2=C\subset X$, if the following normal sequence splits
\begin{align}\label{normalsplit}
  0\arw TC\arw TX|_C\arw N_{C/X}\arw 0,  
\end{align}
then $C\times_{X}C\cong \Tot_C(N_{C/X}[-1])$ and 
$$\Ext^{\bullet}_X(\shO_C, \shO_C)\cong \bigoplus_{p+q=\bullet}\Hdot^q(C, \wedge^p N_{C/X}).$$

\section{Complex Lagrangian Intersections}\label{sec: lagint}
We apply the results in the previous section to the setting of complex Lagrangian intersections. 

Let $(X, \omega)$ be a smooth quasi-projective symplectic variety, i.e., a smooth quasi-projective variety with a closed holomorphic 2-form $\omega\in \Hdot^0(X, \Omega^2)\cong\Hom(\shO_X, \wedge^2 \Omega_X)$ such that the contraction by $\omega$ gives an isomorphism $T_X\cong \Omega_X$. Let $j_{\alpha}:\, C_{\alpha}\hookrightarrow X\,(\alpha=1,2)$ be two smooth (complex) Lagrangians, which means that the previous contraction map induces an isomorphism $N_{C_{\alpha}/X}\cong \Omega_{C_{\alpha}}$.

\subsection{Excess Bundle for Lagrangian Intersections}
\begin{lemma}\label{exc}
    If $C_1$ and $C_2$ are two smooth Lagrangians of $X$ with a clean intersection $B$, then the excess bundle is isomorphic to the cotangent bundle of $B$: $$E\cong  \Omega_{B}.$$
\end{lemma} 
\begin{proof}
It depends on the following local calculations. 
    Let $L_1$ and $L_2$ be two linear Lagrangian subspaces of a symplectic vector space $(V,\omega)$. Let $E$ be the excess space of $L_1\cap L_2$, \emph{i.e.} $E\cong V/(L_1+L_2)$.
    Define a linear map 
    $$\varphi: V\arw (L_1\cap L_2)^{\vee}$$ 
    by $\varphi(v)=\omega(v,-)\in (L_1\cap L_2)^{\vee}$, which is clearly surjective. Then the linear functional $\omega(v,-)$ vanishes for every $w\in L_1\cap L_2$ as long as $v\in L_1+L_2$ by the Lagrangian conditions, which means $\ker(\varphi)\cong L_1+L_2$, and hence $E\cong (L_1+L_2)^{\vee}$.
\end{proof}
The following is observed by Mladenov in \cite{mladenov2024formality}. 
\begin{lemma}\label{lem: canlag}
Let $C_1$ and $C_2$ be two smooth Lagrangians of $X$ with a clean intersection $B$. Then
\begin{align*}
    (K_{C_1}^{\vee}\otimes K_{C_2})|_B\otimes \det(N_{B/C_2})^{\otimes2}\cong \shO_B.
\end{align*}
    where $K_{C_i}$ is the canonical bundle on $C_i$ ($i=1, 2$).
\end{lemma}
\begin{proof}
    By Lemma \ref{excess} and \ref{exc}, there is short exact sequence:
    \begin{align*}
        0\arw N_{B/C_2} \arw N_{C_1/X}|_B\arw \Omega_B\arw 0.
    \end{align*}
    Note that $N_{C_1/X}\cong \Omega_{C_1}$, and then
    \begin{align}\label{eq: can1}
        K_{C_1}|_B\cong K_B\otimes \det(N_{B/C_2}).
    \end{align}
    The determinant of the conormal sequence of $B\subset C_2$ gives
    \begin{align}\label{eq: can2}
        K_{C_2}\cong K_B\otimes \det(N_{B/C_2}^{\vee}).
    \end{align}
    Thus $(K_{C_1}^{\vee}\otimes K_{C_2})|_B\otimes \det(N_{B/C_2})^{\otimes2}\cong \shO_B$ holds by (\ref{eq: can1}) and (\ref{eq: can2}).
\end{proof}

% {\begin{proof} 
% This is a local statement, and we may assume that there is another smooth subvariety $Y\subset X$ such that $C_2\subset Y$, $Y\cap C_1=C_1\cap C_2:=M_{12}$, and $\dim C_1+\dim Y= \dim X-\dim M_{12}$. This can be done by induction on $\codim i_2$.  

% \begin{equation*}
% \begin{tikzcd}
% C_1\cap C_2 \arrow[r,"i_1"] \arrow[d,"i_2"]
% & C_2 \arrow[d,hook', "j_2"] \arrow[r, hook,"i"] &Y \arrow[ld, hook',swap,"j"]\\
% C_1 \arrow[r,hook, swap,"j_1"] 
% & X &.
% \end{tikzcd}
% \end{equation*} 
% By the commutative diagram and the base change formula, we have the following
% $$j_2^{\vee}j_{1*}\shO_{C_1}=i^{\vee}j^{\vee}j_{1*}\shO_{C_1}=i^{\vee}i_*i_{1*}i_2^{\vee}\shO_{C_1}=i^{\vee}i_*i_{1*}\shO_{M_{12}}.$$
% Since $i_*$ is exact, 
% $$i_*\sh\Hdot^{-k}(j_2^{\vee}j_{1*}\shO_{C_1})=\sh\Hdot^{-k}(i_*j_2^{\vee}j_{1*}\shO_{C_1})=\sh\Hdot^{-k}(i_*i^{\vee}i_*i_{1*}\shO_{M_{12}})=Tor_k^Y(\shO_{M_{12}}, \shO_{C_2}).$$
% Write the composition $i\circ i_1$ as $i_{12}$.  
% The normal sequence for the triple $(M_{12}, C_2, Y)$
% $$ 0\arw N_{i_1}\arw N_{i_{12}}\arw N_i|_{M_{12}}\arw 0$$
% implies that $$N_i|_{M_{12}}\cong E. $$
% By projection formula
% $$i_*\shO_{C_2}\overset{\LL}{\otimes} i_{12*} \shO_{M_{12}}=i_{12*}i_1^{\vee}i^{\vee}i_{*}\shO_{M_{2}}=i_{12*}i_1^{\vee}\bigwedge N_i^{\vee}=i_{12*}\bigwedge E^{\vee}$$
% Hence, 
% $$\sh\Hdot^{-k}(j_2^{\vee}j_{1*}\shO_{C_1})\cong \bigwedge^k i_{1*}E^{\vee}$$ holds as required. 
% \end{proof}}
%
%

%
%
% \begin{definition}
% A dga $(K,d)$ is called \emph{formal} if it is quasi-isomorphic to $\bigoplus_{k}^kH ^k(K,d)[-k]$. 
% \end{definition}

%
\subsection{Lagrangian Intersections and Ext Groups}
The excess sequence (\ref{excessplit}) becomes 
$$0\arw TC_1|_B+TC_2|_B\arw TX|_B\arw \Omega_B\arw 0,$$
which is also equivalent to the conormal sequence of $B\subset X$:
\begin{equation}\label{normalseq}
    0\arw N_{B/X}^{\vee}\arw \Omega_X|_B\arw \Omega_B\arw0.
\end{equation}
Based on Theorem \ref{ACH} and Corollary \ref{ssd} , we have:
\begin{proposition}\label{prop: ssd}
If $(C_1, C_2)$ is a pair of smooth Lagrangians in $X$ with a clean intersection $B$, and the conormal sequence (\ref{normalseq}) of $B\subset X$ splits, then 
\begin{align}\label{eq:Hodge=Ext}
    & \Ext^{\bullet}(j_{1*}\shO_{C_1},j_{2*}\shO_{C_2})\cong\bigoplus_{p+q+m=\bullet}
\Hdot^q(B, \Omega_{B}^{p-m}\otimes det(N_{B/C_2})),\\
    & \Ext_X^{\bullet}(j_{1*}F_1,j_{2*}F_2)\cong\bigoplus_{p+q+m=\bullet}
\Hdot^q(B,  \Omega_{B}^{p-m}\otimes (F_1^{\vee}\otimes F_2)|_B\otimes det(N_{B/C_2}))\label{eq:Hodge=Ext2}
\end{align}
where $m=\codim(B, C_2)$, and $F_1$(resp. $F_2$) is a perfect complex on $C_1$(resp. $C_2$).
\end{proposition}
As mentioned in the introduction, there exists cases that $\det(N_{B/C_2})$ or $(F_1^{\vee}\otimes F_2)|_B\otimes \det(N_{B/C_2})$ is torsion so that we can deduce that RHS of (\ref{eq:Hodge=Ext}) or (\ref{eq:Hodge=Ext2}) calculates Hodge cohomology of $B$ in local coefficients. 
\subsubsection{2-torsion local system from spin Lagrangian intersections} Recall that a K\"ahler manifold is spin if and only if its canonical bundle admits square roots. Now suppose that both the Lagrangians $C_1$ and $C_2$ of $X$ are spin, and the square roots of the canonical bundles $K_{C_1}$ of $C_1$ and $K_{C_2}$ of $C_2$ are denoted as $K_{C_1}^{1/2}$ and $K_{C_2}^{1/2}$.
Then Lemma \ref{lem: canlag} gives $(K_{C_1}^{\vee}\otimes K_{C_2})|_B\otimes \det(N_{B/C_2})^{\otimes2}\cong \shO_B$, which implies that $(K_{C_1}^{-1/2}\otimes K_{C_2}^{1/2})|_B\otimes \det(N_{B/C_2})$ determines a 2-torsion $\CC$-local system $\mathfrak{E}$ on $B$.

\subsubsection{Covering trick}
If we assume that $\det(N_{B/C_2})^{\otimes n}\cong \shO_B$, and thus $\det(N_{B/C_2})$ determines a n-torsion rank one $\CC$-local system $\mathfrak{L}:=\det(N_{B/C_2})_{\delta}$ on $B$, then RHS of (\ref{eq:Hodge=Ext}) calculates the $\mathfrak{L}$-twisted Hodge cohomology of $B$. 
% If $\mathfrak{L}$ is trivial, then by Hodge theory,  (\ref{eq:Hodge=Ext}) is isomorphic to $\Hdot^{\bullet-m}(B, \CC)$. If $\mathfrak{L}$ is non-trivial, then there is a double covering $p: \widetilde{B}\arw B$ such that $p_{*}\CC\cong \CC\oplus \mathfrak{L}$. Then by adjointness of $p_*$ and $p^{-1}$
% \begin{align*}
%     \Hdot^{\bullet}(\widetilde{B}, \CC)\cong \Hdot^{\bullet}(B, p_*\CC)=\Hdot^{\bullet}(B, \CC)\oplus \Hdot^{\bullet}(B, \mathfrak{L}).
% \end{align*}
% Similarly we have $p_{*}\shO_{\widetilde{B}}\cong \shO_B\oplus \mathfrak{L}$, and $p^*\Omega_B\cong \Omega_{\widetilde{B}}$. Then 
% \begin{align*}
%     \Hdot^{\bullet}(\widetilde{B}, \Omega^i_{\widetilde{B}})\cong \Hdot^{\bullet}(\widetilde{B}, p^*{\Omega_B})=\Hdot^{\bullet}(B, p_*\shO_{\widetilde{B}}\otimes\Omega^i_B)\cong \Hdot^{\bullet}(B, \Omega^i_B)\oplus \Hdot^{\bullet}(B, \Omega^i_B\otimes\mathfrak{L}).
% \end{align*}
% Together with Hodge theory on $B$ and $\widetilde{B}$, we obtain
% $$\bigoplus_{p+q=\bullet}
% \Hdot^q(B, \Omega_{B}^{p}\otimes \mathfrak{L})\cong \Hdot^{\bullet}(B, \mathfrak{L}).$$
By the covering trick in Chapter 1 of \cite{barth2003compact}, there is a $n$-cyclic covering 
$c: \widetilde{B}\arw B$ such that $$c_{*}\CC\cong \bigoplus_{k=0}^{n-1}\mathfrak{L}^{\otimes k}.$$ Then by adjointness of $c_*$ and $c^{-1}$
\begin{align}\label{eq: coveringcoh}
    \Hdot^{\bullet}(\widetilde{B}, \CC)\cong \Hdot^{\bullet}(B, c_*\CC)=\bigoplus_{k=0}^{n-1}\Hdot^{\bullet}(B, \mathfrak{L}^{\otimes k}).
\end{align}
Similarly, we have $c_{*}\shO_{\widetilde{B}}\cong \bigoplus_{k=0}^{n-1}\mathfrak{L}^{\otimes k}\otimes \shO_B$, and $c^*\Omega_B\cong \Omega_{\widetilde{B}}$. Then 
\begin{align}\label{eq: coveringhodge}
    \Hdot^{\bullet}(\widetilde{B}, \Omega^i_{\widetilde{B}})\cong \Hdot^{\bullet}(\widetilde{B}, c^*{\Omega_B})=\Hdot^{\bullet}(B, c_*\shO_{\widetilde{B}}\otimes\Omega^i_B)\cong \bigoplus_{k=0}^{n-1}\Hdot^{\bullet}(B, \Omega^i\otimes\mathfrak{L}^{\otimes k}).
\end{align}
There is a twisted Hodge-de Rham spectral sequence
\begin{align}\label{eq: twistedhodgederham}
    E_2^{p, q}=\Hdot^{q}(B, \Omega^p\otimes\mathfrak{L}^{\otimes k})\Longrightarrow \Hdot^{p+q}(B, \mathfrak{L}^{\otimes k}) \quad \text{for each }k, 
\end{align}
which implies an inequality
\begin{align}\label{ineq: hodge-derham}
    \dim \Hdot^{\bullet}(B, \mathfrak{L}^{\otimes k})\leq \sum_{p+q=\bullet}\dim \Hdot^{q}(B, \Omega^p\otimes\mathfrak{L}^{\otimes k}).
\end{align}
If $B$ is proper and thus $\tilde{B}$ is also proper, then Hodge theory on $\widetilde{B}$ implies that 
\begin{align}\label{eq: dimhodgederham}
    \dim \Hdot^{\bullet}(\widetilde{B}, \CC)=\sum_{p+q=\bullet} \dim \Hdot^{q}(\widetilde{B}, \Omega^p_{\widetilde{B}}).
\end{align}
Now the combination of the equalities (\ref{eq: coveringcoh}), (\ref{eq: coveringhodge}) and (\ref{eq: dimhodgederham}) forces (\ref{ineq: hodge-derham}) to be an equality for each $k$, and hence the spectral sequence (\ref{eq: twistedhodgederham}) degenerates at $E_2$ page, which implies that 
$$\bigoplus_{p+q=\bullet}
\Hdot^q(B, \Omega_{B}^{p}\otimes \mathfrak{L})\cong \Hdot^{\bullet}(B, \mathfrak{L}).$$
So we have:
\begin{proposition}\label{prop: Ext}
\begin{enumerate} 
\item Let $(C_1, C_2)$ be a pair of Lagrangians of $X$ as in Proposition \ref{prop: ssd}. If $B$ is proper and $\det(N_{B/C_2})$ is torsion, then
$$\Ext_X^{\bullet}(\shO_{C_1},\shO_{C_2})\cong \Hdot^{\bullet-m}(B, \det(N_{B/C_2})_{\delta}),$$ where $m=\codim(B, C_2)$.
\item As a special case, if $C\subset X$ is a proper Lagrangian such that the conormal sequence of $C\subset X$ splits, then $B=C_1=C_2=C$, $\det(N_{B/C_2})_{\delta})\cong\CC$, and 
$$ \Ext^{\bullet}_X(\shO_C, \shO_C)\cong \Hdot^{\bullet}(C,\CC).$$
\end{enumerate}
\end{proposition}
The following is a weaker version of Theorem 0.1.9 in \cite{mladenov2024formality}. 
\begin{proposition}\label{prop: spinlag}
    Let $(C_1, C_2)$ be a pair of Lagrangians of $X$ as in Proposition \ref{prop: ssd}. If both $C_1$ and $C_2$ are spin and $B$ is proper, then 
$$ \Ext^{\bullet}_X(K_{C_1}^{1/2}, K_2^{1/2})\cong \Hdot^{\bullet}(B,(K_{C_1}^{-1/2}|_B\otimes K_{C_2}^{1/2}|_B\otimes \det(N_{B/C_2}))_{\delta}).$$
\end{proposition}
% In particular for the case admissible and proper $C_1=C_2=C\subset X$,
% then
% $$\Ext^{\bullet}_X(\shO_C, \shO_C)\cong \bigoplus_{p+q=\bullet}\Hdot^q(C, \Omega_M^p).$$ If we further assume that $C$ is proper, then by Hodge theory, $$\bigoplus_{p+q=\bullet}\Hdot^q(C, \Omega_C^p)\cong \Hdot^{\bullet}(C,\CC).$$

%
\subsection{Cotangent Bundle}
    Let $X=T^{\vee}M:=\Tot_M(\Omega_M)$ be the total space of the cotangent bundle of a smooth quasi-projective variety $M$, which admits a canonical holomorphic symplectic 2-form $\omega_{can}$. Denote $\pi: X\arw M$ the bundle morphism. There is a canonical 1-form $\theta_{can}$ on $X$, which can be described explicitly. For $(p, q)\in T^{\vee}M$, or $p\in M$ and $q\in T^{\vee}_pM$,
    $$\theta_{can}(p, q)=\pi^{*}q, \,\,\text{and  } \omega_{can}=d\theta_{can}.$$ 
\subsubsection{Zero section Lagrangian} The zero section $j: M\hookrightarrow T^{\vee}M$ is always a Lagrangian of $T^{\vee}M=X$ with split normal sequence. Thus,
    \begin{corollary}\label{cor: Ext}
    The self Ext group of the zero section is the Hodge cohomology: 
       \begin{align}\label{Ext=Hodge}
    \Ext_{T^{\vee}M}^{\bullet}(j_*\shO_M, j_*\shO_M)\cong \bigoplus_{p+q=\bullet} \Hdot^q(M, \Omega_M^p).
\end{align}
If $M$ is proper, then by Hodge theory,
\begin{align}\label{Ext=DR}
    \Ext_{T^{\vee}M}^{\bullet}(j_*\shO_M, j_*\shO_M)\cong \Hdot^{\bullet}(M, \CC).
\end{align} 
    \end{corollary}

\subsubsection{Graph of closed 1-forms} The graph $\Gamma_{\eta}$ of a holomorphic closed 1-form $\eta\in \Hdot^0(M, \Omega)$ is a Lagrangian of $X=T^{\vee}M$.  
\begin{lemma}\label{lem: graph}
   If the zero locus of $\eta$, $Z(\eta)$ is a smooth variety, then $\det(N_{Z(\eta)/M})^{\otimes 2}\cong \shO_{Z(\eta)}$.  
\end{lemma}
 \begin{proof}
 We need to show that $\det(N_{Z(\eta)/M})^{\otimes 2}\cong \shO_{Z(\eta)}$. 
     Let $F$ be a vector bundle on $M$, and $s\in \Hdot^0(M, F)$ be a section. Assume that the zero locus $Z(s)\subset M$ is a smooth variety. Then the differential of the section $Ds|_Z$ gives a morphism $TM|_{Z(s)}\arw F|_{Z(s)}$, which descends to an injective morphism $$N_{Z(s)/M}\arw F|_{Z(s)}.$$
     Now $F=\Omega_M$ and $s=\eta$. Then there is an embedding.  
$$N_{Z(\eta)/M}\hookrightarrow \Omega_M|_{Z(\eta)}.$$
If we compose the above embedding by the natural map $\Omega_M|_{Z(\eta)}\arw \Omega_{Z(\eta)}$, it becomes a zero map. Hence, there is an embedding $N_{Z(\eta)/M}\hookrightarrow N_{Z(\eta)/M}^{\vee} $ which is forced to be an isomorphism by equal ranks of both sides. Therefore, the statement concludes by taking the determinant of them.
 \end{proof}

\begin{corollary}\label{cor:graph}
    Let $\eta$ be a holomorphic closed 1-form on $M$ with smooth zero locus $Z(\eta)$, then 
    $$\Ext_{T^{\vee}M}^{\bullet}(\shO_{\Gamma(\eta)}, \shO_M)\cong \bigoplus_{p+q=\bullet} \Hdot^{q-n}(Z(\eta), \Omega_{Z(\eta)}^p(\mathfrak{L}))),$$
    with $n=\codim(Z(\eta), M)$ and if $Z(\eta)$ is proper, then 
    $$\Ext_{T^{\vee}M}^{\bullet}(\shO_{\Gamma(\eta)}, \shO_M)\cong \Hdot^{\bullet-n}(Z(\eta), \det(N_{Z(\eta)/M})_{\delta}).$$
\end{corollary}
\subsubsection{Conormal varieties}
The conormal bundle $N_{Z/M}^{\vee}$ of a smooth variety $Z\subset M$ is a Lagrangian of $X=T^{\vee}M$. Let $Z_1$ and $Z_2$ be two smooth subvarieties of $M$ with a transverse intersection $Z=Z_1\cap Z_2$. Then $C_1\cap C_2=Z$ and the conormal sequence of $Z\subset X=T^{\vee} M$ splits. 
\begin{corollary}\label{cor: ConormalLag}
For conormal varieties $C_1=N_{Z_1/M}^{\vee}, C_2=N_{Z_2/M}^{\vee}\subset T^{\vee} M$ as above, 
    $$ \Ext^{\bullet}_{T^{\vee}M}(\shO_{C_1}, \shO_{C_2})\cong\bigoplus_{p+q+n=\bullet}
\Hdot^q(Z, \Omega_{Z}^{p-n}\otimes \det(N_{Z_1/M})\otimes \det(N_{Z_2/M}^{\vee})),$$
where $n=\codim (Z, M)$. In particular, if furthermore, $\det(N_{Z_1/M})=\det(N_{Z_2/M})$ and $Z$ is proper, then 
  $$ \Ext^{\bullet}_{T^{\vee}M}(\shO_{C_1}, \shO_{C_2})\cong \Hdot^{\bullet-\codim(Z, M)}(Z, \CC).$$
\end{corollary}
\begin{remark}
    In general, the derived fiber product of two conormal varieties is quasi-isomorphic to the $(-1)$-shifted cotangent space of the derived fiber product $Z_1\times_M Z_2$ by \cite{calaque2019shifted},
    $$C_1\times_X C_2\cong \Tot_{Z_1\times_M Z_2}(\LL_{Z_1\times_M Z_2}[-1]),$$
    which is quasi-isomorphic to $\Tot_{Z_1\cap Z_2}(\Omega_{Z_1\cap Z_2}[-1])$ if $Z_1\cap Z_2$ transversely. 
\end{remark}

%\begin{examples}
%    Let $M$ be a smooth quasi-projective variety and $f: M\arw \CC$ be a holomorphic function. Consider $C_1:=M$ and $C_2:=Graph(df)$, two complex Lagrngians of $X=T^{\vee}M$. Then $C_1\cap C_2=crit(f):=Z$, and 
%    $$ \Ext^{\bullet}(j_{1*}\shO_{C_1},j_{2*}\shO_{C_2})\cong\bigoplus_{p+q+m=\bullet}
%\Hdot^q(Z, \Omega_Z^{p-m}\otimes \omega_{i_1})$$
%holds if the following sequence splits
%$$0\arw N^{\vee}(Z, M)\arw \Omega_M|Z\arw \Omega_Z \arw 0$$
%\end{examples}
%
%
%
%
\section{Equivariant Lagrangian Intersection and Symplectic reduction}\label{sec:eqsym}
We expect that Proposition \ref{prop: ssd} and \ref{prop: Ext} hold for symplectic reduction of $X$ by Hamiltonian $G$-action with $G$ a reductive group. 
\subsection{G-invariant Lagrangians}
Suppose that there is a Hamiltonian $G$-action on $(X, \omega)$ with a holomorphic moment map
$\mu:\, X\longrightarrow \mathfrak{g}^{\vee},$ and let $C\hookrightarrow X$ be a $G$-invariant connected Lagrangian, \emph{i.e.}, $C$ is a connected Lagrangian, and $g\cdot C\subset C$ for every $g\in G$.
We first recall the following fact. 
\begin{lemma}\label{lem: invLag}
  The moment image of $C$ consists of a single central element $c\in \mathfrak{g}^{\vee}$.
\end{lemma}
\begin{proof}
    Let  $\alpha(t)\subset C$ be any curve and $\xi\in \mathfrak{g}$. Denote $\xi^{\#}$ be the generating vector field on $X$. By definition of moment map, $d\mu_{\xi}(\alpha'(t))=\omega(\alpha'(t), \xi^{\#})=0$ since $\alpha'(t)$ is an isotropic vector. Hence, the tangent map $d\mu|_C=0$, which implies that the image of $M$ consists of a single point $c$. Note that $\mu$ is $G$-equivariant, and $C$ is G-invariant, $c$ must  be a central element. 
\end{proof}

\subsection{Symplectic Reductions and Shifted Symplectic Structures} Consider the derived fiber product
 \begin{equation*}
\begin{tikzcd}
X/G\times_{\mathfrak{g}^{\vee}/G} BG \arrow[r] \arrow[d]
& BG=\{*\}/G \arrow[d, "c"] \\
 X/G\arrow[r, "\mu"] & \mathfrak{g}^{\vee}/G.
\end{tikzcd}
 \end{equation*}
We write the derived stack $X/G\times_{\mathfrak{g}^{\vee}/G} BG$ as $X\sslash  G$, called the \emph{symplectic quotient stack} of $X$ by $G$. See \cite{pantev2013shifted}, \cite{anel2022shifted} and \cite{park2025introduction} for the background of shifted symplectic structures, reductions and shifted Lagrangians. 

\begin{proposition}[\cite{pantev2013shifted} and \cite{anel2022shifted}]\label{prop: symred}
    The derived quotient stack $X\sslash  G$ is a $0$-shifted symplectic stack. Moreover, the $G$-invariant Lagrangian $j: C\hookrightarrow X$ induces a (0-shifted) Lagrangian morphism $\alpha: C/G\arw X\sslash  G$, i.e., there is a quasi-isomorphism $\TT_{C/G}\cong \LL_{\alpha}[-1]$, where $\LL_{\alpha}$ is the relative cotangent complex and $\TT_{C/G}$ is the tangent complex of $C/G$. 
\end{proposition}
\begin{proof}
  Note that $\mathfrak{g}^{\vee}/G\cong T^{\vee}[1]BG:=\Tot_{BG}(\TT_{BG}[1])$ is a 1-shifted symplectic stack. The shifted symplectic structure comes from the canonical symplectic structure of the 1-shifted cotangent stack, which is an analogue of the canonical symplectic structure of the cotangent bundle of a smooth variety. 
  
  Both $\mu: X/G\rightarrow \mathfrak{g}^{\vee}/G$ and $c: BG\rightarrow \mathfrak{g}^{\vee}/G$ are Lagrangian morphisms. The first one can be seen by calculating the relative cotangent complex by the distinguished triangle:
  $$\mu^{*}\LL_{\mathfrak{g}^{\vee}/G}\cong[\mathfrak{g}\arw \mathfrak{g^{\vee}}]\arw \LL_{X/G}\cong[\Omega_X\arw \mathfrak{g}^{\vee}]\arw \LL_{\mu}\xrightarrow{[1]},$$
  which implies that $\LL_{\mu}\cong [\mathfrak{g}\arw \Omega_X\cong T_X]\cong \TT_{X/G}.$ Similarly for the latter one. Then the derived fiber product $X\sslash  G$ is a Lagrangian intersection for the 1-shifted symplectic stack $\mathfrak{g}^{\vee}/G$, which admits a 0-shifted symplectic structure by \cite{pantev2013shifted}.  

  The $G$-invariant Lagrangian $C\subset X$ induces a morphism $j: C/G\arw X/G$. Proposition \ref{prop: symred} implies that $\mu\circ j$ factors through the right vertical morphism $c: BG\arw \mathfrak{g^{\vee}}/G$. Then the universal property of derived fiber product gives a unique morphism $\alpha: C/G\arw X\sslash  G$ fitting into the commutative diagram. 
  \begin{equation*}
   \begin{tikzcd}
    C/G 
    \arrow[drr, bend left] \arrow[ddr, bend right, "j"] \arrow[dr, dotted, "\alpha" description] &&\\
    & X\sslash  G \arrow[r, "\tilde{\mu}"] \arrow[d, "\tilde{c}"]
         & BG=\{*\}/G \arrow[d, "c"] \\
 & X/G\arrow[r, "\mu"] & \mathfrak{g}^{\vee}/G.
   \end{tikzcd}
  \end{equation*}
  Notice that there is a distinguished triangle for $C/G\arw X\sslash  G\arw X/G$:
  $$j^{*}\LL_{\tilde{c}}\arw \LL_{j} \arw\LL_{\alpha}\xrightarrow{[1]}$$
  By base-change property, $\LL_{\tilde{c}}=\tilde{\mu}^{\vee}\LL_{c}=\mathfrak{g}[1]$, and by Lagrangian property, $\LL_{j}=T_C[1]$
  Then by combining with the fact that $\TT_{[C/G]}\cong [\mathfrak{g}\arw T_M],$ we have $\TT_{C/G}\cong \LL_{\alpha}[-1]$, and hence the proposition concludes. 
   \end{proof}
\subsection[Alternative Description of X// G]{Alternative Description of $X\sslash G$}
Consider the following derived fiber product of schemes,
 \begin{equation*}
\begin{tikzcd}
X\times_{\mathfrak{g}^{\vee}} \{*\}:=\mu_{der}^{-1}(c) \arrow[r] \arrow[d]
& \{*\} \arrow[d, "c"] \\
 X \arrow[r, "\mu"] & \mathfrak{g}^{\vee}.
\end{tikzcd}
 \end{equation*}
The derived scheme $\mu_{der}^{-1}(c)$ is called the \emph{derived fiber} of $\mu$ over $c\in\mathfrak{g}^{\vee}$, whose classical truncation $(\mu_{der}^{-1}(c))_{cl}=\mu^{-1}(c)$, the ordinary fiber over $c$. Then the $G$-equivalence of $\mu$ induces a natural action of $G$ on $\mu_{der}^{-1}(c)$, and $$X\sslash  G\cong \mu_{der}^{-1}(c)/G.$$
\begin{remark}
    \begin{enumerate}
        \item If $\dim \mu^{-1}(c)=\dim X-\dim G$, then $\mu_{der}^{-1}(c)\cong \mu^{-1}(c)$, and $X\sslash  G$ is just the quotient stack $\mu^{-1}(c)/G$. 
        \item If further $c\in \mathfrak{g}^{\vee}$ is a regular value, then $\mu^{-1}(c)$ is smooth, and the $G$-action on $\mu^{-1}(c)$ is free, and $X\sslash  G=\mu^{-1}(c)/G$ is a smooth symplectic variety, the usual holomorphic Marsden–Weinstein quotient (see \cite{marsden1974reduction}). 
    \end{enumerate}
\end{remark}

%
%
%

%
%

%Let $M_{\alpha}(\alpha=1,2)$ be two smooth $G$-invariant complex Lagrangians of $X$ with a clean intersection $M$. We seek an equality of Proposition \ref{prop: ssd} holds for the symplectic reduction of $X$, and we may assume that both $C_1$ and $C_2$ lie in the same fiber $\mu^{-1}(c)$, where  $c\in Z(\mathfrak{g}^{\vee})\subset \mathfrak{g}^{\vee}$ be a central element. 
\subsection{Invariant Lagrangian intersections and Excess Tangent Complexes}\label{subsec: excess}
Let $C_1$ and $C_2$ be two smooth connected $G$-invariant Lagrangians of $X$ with a nonempty clean intersection $B$. Note that this means that both $C_1$ and $C_2$ have the same moment central value $c\in \mathfrak{g}^{\vee}$ by Lemma \ref{lem: invLag}. Then both $C_1/G$ and $C_2/G$ are Lagrangians of $X\sslash  G$ by Proposition \ref{prop: symred}, and there is a commutative diagram:
%\begin{equation*}
%\begin{tikzcd}
%B/G \arrow[r,"i_2"] \arrow[d,"i_1"]
%& C_2/G \arrow[d,hook', "\alpha_2"] \\
% C_1/G \arrow[r,hook, "\alpha_1"]
%& X\sslash  G.
%\end{tikzcd}
% \end{equation*} 
%
 \begin{equation}\label{diagram: equi}
   \begin{tikzcd}
    B/G 
    \arrow[drr, bend left,"i_2"] \arrow[ddr, bend right, "i_1"] \arrow[dr, dotted, "\gamma" description] &&\\
    & C_1/G\times_{X\sslash  G} C_2/G \arrow[r,"\beta_2"] \arrow[d,"\beta_1"]
& C_2/G \arrow[d,hook', "\alpha_2"] \\
 & C_1/G \arrow[r,hook, "\alpha_1"]
& X\sslash  G.
   \end{tikzcd}
  \end{equation}
  The tangent complexes of the triple $(\beta_2, i_2, \gamma)$ satisfy the following distinguished triangle:
  \begin{align}\label{seq: triple}
      \TT_{\gamma}\arw \TT_{i_2}\arw \gamma^{*}\TT_{\beta_2}\xrightarrow{[1]}.
  \end{align}
  Note that $\TT_{\beta_2}=\beta_1^{*}\TT_{\alpha_1}=\beta_1^{*}\LL_{C_1/G}[-1]$, and $\TT_{i_2}=N_{B/C_2}[-1]$.
  Combine with the excess sequence
  $$ 0\arw N_{B/C_2}\arw N_{C_1/X}\cong \Omega_{C_1}\arw E\cong \Omega_B\arw 0$$
  and $\LL_{B/G}=[\Omega_{C_1}\arw \mathfrak{g}^{\vee}]$, we can deduce that
  $$ \TT_{\gamma}\cong [\Omega_B\arw \mathfrak{g}^{\vee}][-2]\cong \LL_{B/G}[-2].$$
  We may call $\TT_{\gamma}[2]$ the \emph{excess tangent complex} for the above intersection diagram, which can be viewed as a generalization of the excess bundle of intersection of classical schemes. See also Khan \cite{khan2021virtual} for the topic of virtual excess intersection formula. 
In summary, we obtain the following excess distinguished triangle in $B/G$: 
\begin{align}\label{seq: derexcseq}
      \LL_{B/G}[-2] \arw \TT_{i_2}\arw i_1^*\LL_{C_1/G}\xrightarrow{[1]}.
  \end{align}
  
  Define $K_{C/G}:=\det(\LL_{C/G})\cong K_C\otimes \mathfrak{g}^{\vee}$ to be the canonical bundle on the quotient stack $C/G$. Similarly to Lemma \ref{lem: canlag}, we have the following
\begin{proposition}\label{prop: eqcanlag}
Use notations above, we have 
    $$(K^{-1}_{C_1/G}\otimes K_{C_2/G})|_{B/G}\otimes det(\TT_{i_2})^2\cong \shO_{B/G}$$
\end{proposition}
\begin{proof}
    Just as in the proof of Lemma \ref{lem: canlag}, we can conclude the required equation by applying determinant to distinguished triangles \ref{seq: derexcseq} and the relative cotangent complex of $i_1$. 
\end{proof}
%Parallel to Lemma \ref{tor}, we have the following
%\begin{lemma}
%Let $C_1$ and $C_2$ be two smooth $G$-invariant Lagrangians of $X$ with a clean intersection $M$. For each $k$, 
%$$ \sh\Hdot^{-k}(\alpha_2^{*}\alpha_{1*}\shO_{C_1/G})\cong \wedge^k \TT_{B/G} $$
%\end{lemma}  
%
%
%Parallel to Proposition \ref{prop: ss}, we have the following.
%\begin{proposition}
%There is a spectral sequence
%$$E_2^{p,q}=\Hdot^q(B/G, \wedge^{p-m}\LL_{B/G}\otimes det(N_{B/C_2}^{\vee}))\Longrightarrow \Ext^{p+q}_{X\sslash  G}(\alpha_{1*}\shO_{C_1/G},\alpha_{2*}\shO_{C_2/G}),$$
%where $m=\dim C_2-\dim M$.
%\end{proposition}
%
From now on, we will use the following set-up, mentioned in the introduction. 
\begin{setup}[S]
$(X,\omega)$ is a smooth symplectic quasi-projective variety with a Hamiltonian $G$-action, where $G$ is a complex linear reductive group, and $C_1, C_2$ are $G$-invariant smooth connected Lagrangians of $(X, \omega)$ such that the scheme theoretic intersection $B:=C_1\cap C_2$ is smooth.
\end{setup}
Recall that the conormal sequence of $G$-invariant subscheme $B\subset X$ splits if the following conormal bundle sequence splits as $G$-equivariant bundles
\begin{align*}
    0\arw N^{\vee}_{B/X}\arw \Omega_X|_B \arw \Omega_B\arw 0. 
 \end{align*}

\begin{theorem}\label{mainthm}
In Set-up (S), if the conormal sequence of $B\subset X$ splits, then 
$$C_1/G\times_{X\sslash  G} C_2/G\cong \Tot_{B/G}(\LL_{B/G}[-1])$$
as derived schemes over $C_1/G\times C_2/G$.
\end{theorem}
\begin{proof}
Following the idea of \cite{arinkin2012self} and \cite{arinkin2019formality}, we consider the following commutative diagram for derived schemes. 
\begin{equation*}
   \begin{tikzcd}
    W 
    \arrow[drr, bend left,"p"] \arrow[ddr, bend right, "q"] \arrow[dr, "\gamma"] &&\\
    & T\times_S U\arrow[r,"\tilde{j}"] \arrow[d,"\tilde{i}"]\arrow[ul, dotted, shift left=2, "\pi"]
& T \arrow[d, "i"] \\
 & U \arrow[r, "j"]
& S.
   \end{tikzcd}
  \end{equation*}
There is a distinguished triangle of tangent complexes:
$$\TT_{\gamma}\arw \TT_q\arw \gamma^{*}\TT_{\tilde{i}}\cong p^{*}\TT_i\xrightarrow{[1]}$$
Assume that there is a splitting $\TT_{\gamma}[1]\arw p^{*}\TT_i$ which is a right inverse to the connecting morphism $p^{\vee}\TT_i\arw \TT_{\gamma}[1]$. 
Assume also that  $T\times_S T\cong \Tot_S(\TT_i)$. Then we have a morphism
$$ \varphi: \Tot(\TT_{\gamma}[1])\arw \Tot(p^{*}\TT_i)\cong T\times_S T\times_T W=T\times_S W\arw T\times_S U,$$
where the last arrow comes from the morphism $q$. 
To check $\varphi$ is an isomorphism, it is sufficient to check locally. 

Now we apply the idea to current case. Let $i$ be the diagonal morphism $\Delta: T=\mu_{der}^{-1}(c) \arw S=\mu_{der}^{-1}(c)$, $U$ be $C_1\times C_2$ and $j$ be the morphism induced by $\alpha_1, \alpha_2$. Then the derived fiber product $T\times_S U\cong C_1\times_{\mu_{der}^{-1}(c)} C_2$ and $W=M$. It is known that $\Delta^{*}\Delta_* \shO_{\mu_{der}^{-1}(c)}\cong \Sym (\LL_{\mu_{der}^{-1}(c)}[1])=\Sym (\LL_{\Delta})$ by \cite[Proposition 1.1]{ben2012loop}. The distinguished triangle is
$$ \TT_{\gamma}\arw (TC_1+TC_2)|_B[-1] \arw \TT_{\mu_{der}^{-1}(c)}\cong [T_X|_B\arw \mathfrak{g}^{\vee}][-1] \xrightarrow{[1]}$$
and the splitting morphism is induced by the splitting $\Omega_M\arw TX|_B$. Hence, there is a morphism
$\varphi: \Tot(\LL_{B/G}[1])\arw C_1\times_{\mu_{der}^{-1}(c)} C_2$. Then $\varphi$ is isomorphic by a local calculation as shown below.   

We may assume that all $X=\Spec\,R, C_1=\Spec R/I, C_2=\Spec R/J, M=\Spec \frac{R}{I+J}$ are regular affine schemes, where $R$ is a regular local $\CC$-algebra, both $I$ and $J$ are ideals of $R$, generated by regular sequences. As an $R/J$-algebra, $\frac{R}{I+J}=\dfrac{R/J}{I+J/J}$. The regularity of $M$ implies that the image of $I$ in $R/J$ is generated by a regular sequence $g_1, g_2, \cdots, g_k$, which has a lift to a regular sequence (denoted still) $g_1, g_2, \cdots, g_k$. Consider the quotient ring $R/(g_1, \cdots, g_k)$, and its ideal $I/(g_1,\cdots, g_k)$. By \cite{vasconcelos1967ideals}, $I/(g_1,\cdots, g_k)$ is generated by a regular sequence $(g_{k+1}, g_2,\cdots, g_n)$, and hence $I$ is generated by the regular sequence $(g_1, \cdots, g_n)$. Then $R/I$ is quasi-isomorphic to the Koszul dg algebra $R[e_1, e_2, \cdots, e_{n}]$, where $deg(e_i)=-1, d(e_i)=g_i$. Now $R/I\otimes_R R/J\cong R/J[e_1, \cdots, e_n]\cong \frac{R}{I+J}[e_{k+1}, \cdots, e_n]$. Note that $d(e_i)=0$, and hence $R/I\otimes_R R/J\cong \Sym_{\frac{R}{I+J}}(E^{\vee}[1])$, where $E^{\vee}$ is the free $\frac{R}{I+J}$-module generated by $e_{k+1},\cdots, e_n$ which is dual excess bundle on $\Spec(\frac{R}{I+J})$.

$\mu^{-1}(c)=\Spec R/A,$ s and $A\subset I\cap J$
The moment map $\mu$ gives a sequence of elements $a_1, a_2, \cdots, a_l$ in $R$ ($l=\dim G$). Consider the ideal $A=(a_1, \cdots, a_l)$. Lemma \ref{lem: invLag} implies that $A\in I\cap J$, and we can write each $a_i=\sum_{j}c_{ij}g_j$. Note that $\mu_{der}^{-1}(c)$ is an affine derived scheme determined by the Koszul dg algebra $R[f_1, \cdots, f_k]$, where $deg (f_i)=-1, d(f_i)=a_i$. Then the Tate algebra $R[e_1, \cdots, e_n,\epsilon_1, \cdots, \epsilon_l, f_1, \cdots, f_k]$ is a dg resolution of $R/I$ as $R[f_1, \cdots, f_k]$-algebra, where $deg(\epsilon_i)=-2, d(\epsilon_i)=\sum_j c_{ij}e_j-f_i$. 
 
 Then $R/I\otimes_{R[f_1, \cdots, f_k]} R/J=\Sym_{\frac{R}{I+J}}(E^{\vee}[1])[\epsilon_1, \cdots, \epsilon_t]$, where $d(\epsilon_i)=\sum_j c_{ij}e_j$, which means that $R/I\otimes_{[f_1, \cdots, f_k]} R/J=\Sym_{\frac{R}{I+J}}([\mathfrak{g}\arw E^{\vee}][1])=\Tot(\TT_{\gamma}[1])$. Therefore, the morphism $\varphi$ is an isomorphism. 
\end{proof}

%\begin{remark}
 %   Theorem \ref{mainthm} gives a canonical $(-1)$-shifted symplectic structure on $C_1/G\times_{X\sslash  G} C_2/G$. 
%\end{remark}
\subsection{Ext Groups of Invariant Lagrangians and Equivariant Cohomology}
% If $C$ is a smooth proper $G$-invariant Lagrangian of $X$ such that the Lagrangian pair $(C, C)$ satisfies Assumption Set-up (S), then 
% $$\shO_{C/G}\otimes_{X\sslash  G}\shO_{C/G}\cong \Tot_{C/G}(\LL_{C/G}[-1]),$$
% or equivalently, there is a quasi-isomorphism
% $$\alpha^{*}\alpha_*\shO_{C/G}\cong \bigoplus_{p\geq 0}\bigwedge^p \TT_{C/G}[p],$$
% where $\alpha: C/G\hookrightarrow X/G$ is the embedding and $\bigwedge^p(-)$ denotes the derived $p$-th exterior power.
% By combining this with the adjunction pair $\alpha^{*}\dashv \alpha_{*}$, we can obtain the following
% \begin{align*}
% \Ext_{X\sslash  G}^{\bullet}(\alpha_{*}\shO_{C/G},\alpha_{*}\shO_{C/G})\cong \Ext_{C/G}^{\bullet}(\alpha^{*}\alpha_{*}\shO_{C/G},\shO_{C/G})\cong \bigoplus_{p\geq 0} \Hdot^{\bullet-p}(C/G, \bigwedge^p \LL_{C/G}),
% \end{align*}
% which calculates the Hodge cohomology of the stack $B/G$.
Let $(C_1, C_2)$ be a pair of $G$-invariant Lagrangians of $X$ in Set-up (S) and  the conormal sequence of $B\subset X$ splits, then
Theorem \ref{mainthm} gives 
$$C_1/G\times_{X\sslash  G} C_2/G\cong \Tot_{B/G}(\LL_{B/G}[-1])$$
over derived stacks $C_1/G\times C_2/G$. Thus the natural morphism $\pi: \Tot_{B/G}(\LL_{B/G}[-1])\arw B/G$ fits into  (\ref{diagram: equi}) such that $\beta_1=i_1\circ \pi$ and $\beta_2=i_2\circ \pi$. Also
$$\pi_*\shO_{\Tot_{B/G}(\LL_{B/G}[-1])}\cong \bigoplus_{p\geq 0}\bigwedge^p \TT_{B/G}[p],$$
where $\bigwedge^p(-)$ denotes the derived $p$-th exterior power.
Then we can derive as in subsection \ref{subsec: derint} by the adjunction of pullback-pushforward and base-change property. 
\begin{align*}
    &\Ext_{X\sslash G}^{\bullet}(\alpha_{1*}\shO_{C_1/G}, \alpha_{2*}\shO_{C_2/G})\cong \Ext_{C_2/G}^{\bullet}(\alpha_2^{*}\alpha_{1*}\shO_{C_1/G}, \shO_{C_2/G})
    \cong \Ext_{C_2/G}^{\bullet}(\beta_{2*}\beta_1^{*}\shO_{C_1/G}, \shO_{C_2/G})\\
    &\cong \Ext_{C_2/G}^{\bullet}(i_{2*}\pi_*(\shO_{\Tot_{B/G}(\LL_{B/G}[-1])}), \shO_{C_2/G})\cong \Ext_{B/G}^{\bullet}(\bigoplus_{p\geq 0}\bigwedge^p\TT_{B/G}[p], i_2^{!}\shO_{C_2/G}).
\end{align*}
The last isomorphism comes from the Grothendieck duality,
and hence, 
\begin{align}
    \Ext^{\bullet}_{X\sslash  G}(\shO_{C_1/G},\shO_{C_2/G})\cong\bigoplus_{p+q+m=\bullet}
\Hdot^q(B/G, \wedge^{p-m}\LL_{B/G}\otimes \det(N_{B/C_2}^{\vee})),
\end{align}
where $m=\codim(B, C_2)$.

Similarly, if $F_1$ and $F_2$ are $G$-equivariant perfect complexes on $C_1$ and $C_2$ respectively, then
\begin{align}\label{eq: generalexteq}
    \Ext^{\bullet}_{X\sslash  G}(F_1,F_2)\cong\bigoplus_{p+q+m=\bullet}
\Hdot^q(B/G, (F_1^{\vee}\otimes F_2)|_{B/G}\otimes \wedge^{p-m}\LL_{B/G}\otimes \det(N_{B/C_2}^{\vee})),
\end{align}
where $m=\codim(B, C_2)$.
In particular, when $C_1=C_2=C$, then 
\begin{align*}
\Ext_{X\sslash  G}^{\bullet}(\shO_{C/G},\shO_{C/G})\cong  \bigoplus_{p\geq 0} \Hdot^{\bullet-p}(C/G, \bigwedge^p \LL_{C/G}),
\end{align*}
which calculates the Hodge cohomology of the quotient stack $B/G$.

If, moreover, the Hodge-to-de Rham spectral sequence for the stack $B/G$ degenerates at the $E_1$ page, then Hodge cohomology calculates the equivariant cohomology $\Hdot^{\bullet}_G(B, \CC)$. The degeneracy of the Hodge to de Rham spectral sequence can be guaranteed by the condition that $B/G$ is \emph{cohomologically Hodge spreadable} in \cite{kubrak2022hodge}. For example, if $h^0(B, \shO)^G<\infty$ and $B$ is proper over affine, then $B/G$ is cohomologically Hodge spreadable (see \cite[Example 3.1.6]{kubrak2022hodge}) . Therefore, we have 
\begin{corollary}\label{symredext}
    Let $C\hookrightarrow X$ be a $G$-invariant smooth Lagrangian of $X$ whose conormal sequence splits. Assume that $h^0(C, \shO)^G<\infty$ and $C$ is proper over affine, then 
    $$ \Ext_{X\sslash  G}^{\bullet}(\shO_{C/G}, \shO_{C/G})\cong \Hdot^{\bullet}_G(C,\CC).$$
\end{corollary}

% Theorem \ref{mainthm} implies that 
% $$\alpha_2^{*}\alpha_{1*}\shO_{C_1/G}\cong \Sym(\TT_{B/G}[1])=\bigoplus_{p\geq 0}\wedge^p \TT_{B/G},$$

% By combining this with the adjunction pair $\alpha_2^{*}\dashv \alpha_{2*}$, we can obtain the following

% $$ \Ext^{\bullet}_{X\sslash  G}(\alpha_{1*}\shO_{C_1/G},\alpha_{2*}\shO_{C_2/G})\cong\bigoplus_{p+q+\codim(M, C_2)=\bullet}
% \Hdot^q(B/G, \wedge^{p-\codim(M, C_2)}\LL_{B/G}\otimes det(N_{B/C_2}^{\vee})).$$

Mimicking the covering trick in the proof of Proposition \ref{prop: Ext}, we have:
\begin{corollary}\label{symredext2} In Set-up (S), and assume that the conormal sequence of $B\subset X$ splits. Then 
   $$ \Ext^{\bullet}_{X\sslash  G}(\shO_{C_1/G},\shO_{C_2/G})\cong\bigoplus_{p+q+m=\bullet}
\Hdot^q(B/G, \wedge^{p-m}\LL_{B/G}\otimes \det(N_{B/C_2}^{\vee})), $$
where $m=\codim(B, C_2)$. 
If furthermore $\det(N_{B/C_2})$ is torsion in $\Pic_G(B)$, $B$ is proper over affine, and $h^0(B, \shO)^G<\infty$, then 
$$\Ext^{\bullet}_{X\sslash  G}(\shO_{C_1/G},\shO_{C_2/G})\cong \Hdot^{\bullet-m}_G(B, \det(N_{B/C_2})_{\delta}).$$

\end{corollary}
\begin{definition}[Equivariant spin]
    We call a $G$-invariant subvariety $C\subset X$ is \emph{equivariant spin} if the canonical bundle $K_{C/G}$ has a square root $K_{C/G}^{1/2}$ in $\Pic_G(C)$.
\end{definition}  
Then Proposition \ref{prop: eqcanlag} implies that $K_{C_1/G}^{-1/2}|_{B/G}\otimes K_{C_2/G}^{1/2}|_{B/G}\otimes \det(N_{B/C_2})$ is 2-torsion in $B/G$ if both $C_1$ and $C_2$ are equivariant spin. Then by covering tricks,  
the RHS of (\ref{eq: generalexteq}) calculates the local equivariant cohomology of $B$ when $F_i=K_{C_i/G}^{1/2} (i=1,2)$. 
 \begin{theorem}\label{thm: can}
    In Set-up (S), assume that the conormal sequence of $B\subset X$ splits. Assume also $K_{C_i/G}^{1/2}$ exist square $(i=1,2)$. If $h^0(B, \shO)^G<\infty$ and $B$ is proper over affine, then
    $$\Ext^{\bullet}_{X\sslash G}(K^{1/2}_{C_1/G}, K^{1/2}_{C_2/G})=\Hdot^{\bullet-\codim(B, C_2)}_G(B, \mathfrak{E}),$$
    where $\mathfrak{E}$ is the $G$-equivariant $2$-torsion $\CC$-local system determined by $K_{C_1/G}^{-1/2}|_{B/G}\otimes K_{C_2/G}^{1/2}|_{B/G}\otimes \det(N_{B/C_2})$.
\end{theorem}
\subsection{Cotangent Bundle Cases}
Given any smooth quasi-projective $G$-variety $M$, its cotangent bundle $X=T^{\vee}M$ has a natural Hamiltonian $G$-action, and its moment map $\mu: X\arw \mathfrak{g^{\vee}}$ can be described explicitly:
$$\mu_{\xi}(p, q)=\langle q, \xi^{\#}(p)\rangle,$$
where $p\in M, q\in T^{\vee}_pM$, and $\xi^{\#}(p)$ is the vector field on $M$ generated by $\xi\in \mathfrak{g}$ at $p$. 
\subsubsection{Zero section Lagrangians}
The zero section $M\subset X$ is always a $G$-invariant Lagrangian whose moment map value equals $0$, and the conrmal sequence of $M\subset X$ splits. 
\begin{corollary}\label{cor:eqExt}
    The $G$-equivariant Ext group of the zero section by itself is isomorphic to the Hodge cohomology of the quotient stack: 
\begin{align}\label{GEXT=H1}
    \Ext_{T^{\vee}M\sslash  G}^{\bullet}(\shO_{M/G}, \shO_{M/G})\cong \bigoplus_{p+q=\bullet}\Hdot^{q}(M/G, \wedge^p \LL_{M/G}).
\end{align}
If $h^0(M, \shO)^G<\infty$ and $M$ is proper over affine, then
\begin{align}\label{GEXT=H2}
    \Ext_{T^{\vee}M\sslash  G}^{\bullet}(\shO_{M/G}, \shO_{M/G})\cong \Hdot^{\bullet}_G(M, \CC),
\end{align}
    \end{corollary}
\subsubsection{Graph of invariant holomorphic 1-form} The graph $\Gamma_{\eta}$ of a $G$-invariant holomorphic closed 1-form $\eta\in \Hdot^0(M, \Omega)^G$ is a $G$-invariant Lagrangian of X. Similarly to Lemma \ref{lem: graph}, when the zero locus $Z(\eta)$ is smooth, its normal bundle in $M$ is 2-torsion.
\begin{corollary}\label{cor: eqgraph}
    Let $\eta$ be an $G$-invariant holomorphic closed 1-form on $M$ with smooth zero locus, then 
    $$\Ext_{T^{\vee}M\sslash  G}^{\bullet}(\shO_{\Gamma_{\eta}/G}, \shO_{M/G})\cong \bigoplus_{p+q=\bullet} \Hdot^{q-n}(Z(\eta)/G, \wedge^p\LL_{Z(\eta)/G}),$$
    with $n=\codim(Z(\eta), M)$.
    If $Z(\eta)$ is proper, then 
    $$\Ext_{T^{\vee}M\sslash  G}^{\bullet}(\shO_{\Gamma(\eta)/G}, \shO_{M/G})\cong \Hdot^{\bullet-n}_G(Z(\eta), (N_{Z(\eta)/M})_{\delta}),$$
\end{corollary}
\subsubsection{Equivariant conormal varieties}Let $Z_1$ and $Z_2$ be two smooth $G$-invariant subvarieties of $M$ with a transverse intersection $Z=Z_1\cap Z_2$. Consider conormal varieties $C_1=N_{Z_1/M}^{\vee}$ and $C_2=N_{Z_2/M}^{\vee}$,  which are $G$-invariant Lagrangians of $X=T^{\vee}M$ whose intersections $Z$ has splitting conormal sequence. Hence,
\begin{corollary}\label{cor: eqconormal}
    For conormal varieties $C_1, C_2$ as above, 
    $$ \Ext^{\bullet}_{X\sslash  G}(\shO_{C_1/G}, \shO_{C_2/G})\cong\bigoplus_{p+q+n=\bullet}
\Hdot^q(Z/G, \wedge^{p-n}\LL_{Z/G}\otimes \det(N_{Z_1/M})\otimes \det(N_{Z_2/M}^{\vee})),$$
where $n=\codim (Z, M)$.
In particular, if furthermore, $\det(N_{Z_1/M})=\det(N_{Z_2/M})$ and $Z$ is proper, then 
  $$ \Ext^{\bullet}_{X\sslash  G}(\shO_{C_1/G}, \shO_{C_2/G})\cong \Hdot^{\bullet-n}_G(Z, \CC).$$
  \end{corollary}
  \begin{remark}
      By Anel and Calaque \cite{anel2022shifted}, $T^{\vee}M\sslash_0  G\cong \Tot_{M/G}(\LL_{M/G})$, which can be viewed as the "cotangent bundle of the quotient stack M/G". Then Corollary \ref{cor:eqExt}, \ref{cor: eqgraph} and \ref{cor: eqconormal} are $G$-equivariant analogue of Corollary \ref{cor: Ext}, \ref{cor:graph} and \ref{cor: ConormalLag}, respectively.
  \end{remark}

\section{Geometric Invariant Theory, Semi-stable symplectic quotient and Kirwan Surjectivity}\label{sec:git}
\subsection{Geometric Invariant Theory}
Let $M$ be a smooth complex projective over affine variety with an algebraic reductive group $G$-action on it. We denote $L$ a linearization of the action, i.e., a $G$-equivariant ample line bundle on $M$, which defines a $G$-invariant open \emph{semi-stable} locus 
$$M^{ss}(L):=\{q|\text{ there exists } f\in \Hdot^0(M, L^k)^G \text{ such that}\,f(q)\neq 0 \text{ for some } k> 0\}\subset M.$$ Then we may call the quotient stack $M^{ss}(L)/G$ the \emph{semi-stable quotient stack} or \emph{GIT quotient stack}. We may also write the semi-stable quotient as $M^{ss}/G$ when the linearization is obvious. 

\subsection{Semi-stable Symplectic Quotients}
Assume that the Hamiltonian $G$-action of $(X, \omega)$ has a linearization, \emph{i.e.,} there is a $G$-equivariant ample line bundle $L$ on $X$. We define the \emph{semi-stable symplectic quotient} as the symplectic quotient of the semi-stable locus $X^{ss}\sslash  G$. It is easy to see that $X^{ss}\sslash  G=X^{ss}/G\times_{X/G} X\sslash  G$.
If $j: C\subset X$ is a $G$-invariant Lagrangian, then $j^{*}L$ gives a linearization of $(C, G)$, and $C^{ss}(j^{*}L)=X^{ss}(L)\cap C$.
\subsubsection{Cotangent bundle}
  Let $(M, L, G)$ be a smooth projective over affine $G$-variety with a linearization $L$. Then $X=T^{\vee}M$ admits a natural $G$-action such that the bundle morphism $\pi: X\arw M$ and the zero section $j: M\arw X$ are both $G$-equivariant. The linearization on $X$ is given by $\pi^{*}L$ due to the following proposition. 
\begin{proposition}
    $\pi^{*}L$ is ample on $X$.
\end{proposition}
  \begin{proof}
      By Serre's theorem, $T_M\otimes L^k$ is globally generated for sufficiently large $k$'s. Then there is a vector space $V$ such that $V\otimes\shO\arw T_M\otimes L^k$ is surjective, and we have an embedding $\Omega_M\hookrightarrow V^{\vee}\otimes L^{-k}$, which gives an embedding for the total spaces:
      $$T^{\vee}M\hookrightarrow \Tot_M(V^{\vee}\otimes L^{-k}).$$
      We can further embed $\Tot_M(V^{\vee}\otimes L^{-k})$ into the projective bundle 
      $\PP_M(V^{\vee}\otimes L^{-k}\oplus \shO).$
      Since $L^k$ is globally generated for sufficiently large $k$'s, there is a vector space $W$, such that $L^{-k}\hookrightarrow W\otimes \shO$, and we can embed $\PP_M(V^{\vee}\otimes L^{-k}\oplus \shO)$ into $\PP_M(V^{\vee}\otimes W\oplus \shO)\cong M\times \PP(V^{\vee}\otimes W\oplus\CC)$
       Then $L':=L\boxtimes \shO_{\PP(V^{\vee}\otimes W\oplus\CC)}(1)$ is an ample line bundle on $M\times \PP(V^{\vee}\otimes W\oplus\CC)$, whose restriction on $T^{\vee}M$ is $\pi^*L$. 
  \end{proof}
  \begin{proposition} 
      $(T^{\vee}M)^{ss}\sslash  G$ contains $T^{\vee}({M^{ss}/G})$ as an open substack. 
  \end{proposition}
\begin{proof}
    By \cite{anel2022shifted}, $T^{\vee}({M^{ss}/G})\cong T^{\vee}(M^{ss})\sslash  G$. It is sufficient to show that $T^{\vee}(M^{ss})\subset (T^{\vee}M)^{ss}$ is an open subvariety. Pick up any $(p, q)\in T^{\vee}(M^{ss})$, there is a $f\in \Hdot^0(M, L^k)^G$ such that $f(p)\neq 0$ for some $k> 0$. Then $\pi^*f\in \Hdot^0(T^{\vee}M, \pi^*L^k)^G$ and $(\pi^*f)(p, q)=f(p)\neq 0,$ which implies that $(p, q)\in (T^{\vee}M)^{ss}.$ Hence, $T^{\vee}(M^{ss})\subset (T^{\vee}M)^{ss}.$ Note that semi-stable locus $M^{ss}\subset M$ is an open subvariety, and so is $T^{\vee}M^{ss}\subset T^{\vee}M$, and hence $T^{\vee}(M^{ss})\subset (T^{\vee}M)^{ss}$ is open since $(T^{\vee}M)^{ss}\subset T^{\vee}M$ is open.
\end{proof}
The next example illustrates that $T^{\vee}({M^{ss}/G})$ may not equal to $ (T^{\vee}M)^{ss}\sslash  G$.
\begin{examples}
Let $M=\CC^2$, which admits a $G=\CC^{*}$ action of weight $(1,-1)$. Denote a general point of $M$ by $(z_1, z_2)$, and the action is given by $t.(z_1, z_2)=(tz_1, t^{-1}z_2)$ for $t\in \CC^*$. Denote a general point of $X=T^{\vee}M=\CC^4$ by $(z_1, z_2, w_1, w_2)$, and the induced $\CC^*$ action has weight $(1,-1,1,-1)$. 

Choose a linearization $L=\shO(\chi)$ on $M$, where $\chi: \CC^{*}\arw \CC^{*}$ is the identity character. Then $M^{ss}=\{z_1\neq 0\}$, and $(T^{\vee}M)^{ss}=\{z_1\neq 0\}\cup\{w_2\neq 0\}$ which contains but not equal to $T^{\vee}(M^{ss})=\{z_1\neq 0\}$. Their symplectic quotients
$$T^{\vee}(M^{ss})\sslash  G\cong T^{\vee}(M^{ss}/G)=T^{\vee}\CC \subsetneq T^{\vee}\PP^1\cong (T^{\vee}M)^{ss}\sslash  G.$$
\end{examples}

\subsection{Twisted Kirwan Surjectivity}
The natural embedding $M^{ss}\hookrightarrow M$ induces a map on their $G$-equivariant cohomology, which is called the \emph{Kirwan map}. The Kirwan surjectivity states that the Kirwan map is surjective on the torsion free part \cite{kirwan1984cohomology}:
$$ Kir: \Hdot^{\bullet}_G(M,
\CC)\longrightarrow \Hdot^{\bullet}_G(M^{ss}, \CC).$$ 
If there is a $G$-equivariant $\CC$-local system $\mathfrak{L}$ on $B$, then the following twisted Kirwan map is also surjective (Theorem \ref{thm:twistedkirwan}):
$$ Kir_{\mathfrak{L}}: \Hdot^{\bullet}_G(B,
\mathfrak{L})\longrightarrow \Hdot^{\bullet}_G(B^{ss}, \mathfrak{L}).$$ 

%\begin{proof}
%Apply KN stratification + equivariant MV seq....
%\end{proof}

\subsection{Application of Twisted Kirwan Surjectivity}
We can apply the results in Section \ref{sec:eqsym} to the symplectic quotient $X^{ss}\sslash  G$. It is easy to see that the GIT quotient stack $C^{ss}/G$ is a Lagrangian of $X^{ss}\sslash  G$ if $C$ is a $G$-invariant Lagrangian, and ($C_1^{ss}/G, C_2^{ss}/G)$ satisfies Set-up (S) and conormal $B\subset X$ splits if $(C_1/G, C_2/G)$ does. By Example 3.1.7 of \cite{kubrak2022hodge}, the Hodge to de Rham degeneracy of $B^{ss}/G$ holds if $B$ is proper over affine and $\dim H^0(B, \shO)^G<\infty$. So we have:
\begin{proposition}\label{prop: gitext}
   In Set-up (S), assume that conormal $B\subset X$ splits. Let $L$ be an $G$ ample line bundle on $X$. Then
$$ \Ext^{\bullet}_{X^{ss}\sslash  G}(\shO_{C^{ss}_1/G},\shO_{C^{ss}_2/G})\cong\bigoplus_{p+q+m=\bullet}
\Hdot^q(B^{ss}/G, \wedge^{p-m}\LL_{B^{ss}/G}\otimes det(N_{B^{ss}/C_2^{ss}}^{\vee})), $$
where $m=\codim(B, C_2)$.
If furthermore $\det(N_{B/C_2})$ is torsion in $\Pic(B/G)$, $B$ is proper over affine and $\dim H^0(B, \shO)^G<\infty$, then 
$$\Ext^{\bullet}_{X^{ss}\sslash  G}(\shO_{C^{ss}_1/G},\shO_{C^{ss}_2/G})\cong \Hdot^{\bullet-m}_G(B^{ss}, \det(N_{B^{ss}/C^{ss}_2})_{\delta}).$$

\end{proposition}
Moreover, the embedding $X^{ss}\sslash  G\hookrightarrow X\sslash  G$ induces an $\infty$-functor by pullback
$$res: D_{qc}(X\sslash  G)\arw D_{qc}(X^{ss}\sslash  G).$$
Note that all isomorphisms in Theorem \ref{mainthm} and Corollary \ref{symredext}, \ref{symredext2} are natural, and we have the following statements.

\begin{theorem}\label{mainthm2}
In Set-up (S), assume that conormal $B\subset X$ splits. Let $L$ be an $G$ ample line bundle on $X$. Suppose that $\det(N_{B/C_2})$ is torsion, $B$ is proper over affine, and $\dim H^0(B, \shO)^G<\infty$.  We have the following commutative diagram
\begin{equation*}
\begin{tikzcd}
     \Ext_{X\sslash  G}^{\bullet}(\shO_{C_1/G}, \shO_{C_2/G}) \arrow[rr, "res"] \arrow[d, "\cong"] && \Ext^{\bullet}_{X^{ss}\sslash  G}(\shO_{C_1^{ss}/G}, \shO_{C_2^{ss}/G})\arrow[d, "\cong"]\\
     \Hdot^{\bullet-m}_G(B, \mathfrak{L}) \arrow[rr, "Kir_{\mathfrak{L}}"]&& \Hdot^{\bullet-m}_G(B^{ss}, \mathfrak{L}),
\end{tikzcd}
\end{equation*}
where $m=\codim(B, C_2)$ and $\mathfrak{L}:=\det(N_{B/C_2})_{\delta}$. Both $Kir_{\mathfrak{L}}$ and $res$ are surjective by Theorem \ref{thm:twistedkirwan}.
\end{theorem}

In particular, when $C_1=C_2=C$, we have 
\begin{theorem}\label{thm: ssc}
    Let $C$ be a $G$-invariant Lagrangian of $X$ such that the conormal sequence of $C\subset X$ splits, and suppose further that $C$ is proper over affine and $\dim H^0(C, \shO)^G<\infty$. Then there is a commutative diagram
\begin{equation*}
\begin{tikzcd}
     \Ext_{X\sslash  G}^{\bullet}(\shO_{C/G}, \shO_{C/G})  \arrow[d, "\cong"] \arrow[rr, "res"]&& \Ext^{\bullet}_{X^{ss}\sslash  G}(\shO_{C^{ss}/G}, \shO_{C^{ss}/G})\arrow[d, "\cong"]\\
     \Hdot^{\bullet}_G(C, \CC) \arrow[rr, "Kir"]&& \Hdot^{\bullet}_G(C^{ss}, \CC).
\end{tikzcd}
\end{equation*}
The Kirwan map $Kir$ is surjective and hence $res$ is also surjective. 
\end{theorem}

\begin{theorem}\label{thm: kircan}
    In Set-up (S), assume that conormal $B\subset X$ splits. Let $L$ be an $G$ ample line bundle on $X$. Suppose that $K_{C_1/G}^{1/2}$ and $K_{C_2/G}^{1/2}$ exist, $B$ is proper over affine and $\dim H^0(B, \shO)^G<\infty$. We have the following commutative diagram
\begin{equation*}
\begin{tikzcd}
     \Ext_{X\sslash  G}^{\bullet}(K_{C_1/G}^{1/2}, K_{C_2/G}^{1/2}) \arrow[rr, "res"] \arrow[d, "\cong"] && \Ext^{\bullet}_{X^{ss}\sslash  G}(K_{C^{ss}_1/G}^{1/2}, K_{C^{ss}_2/G}^{1/2})\arrow[d, "\cong"]\\
     \Hdot^{\bullet-m}_G(B, \mathfrak{E}) \arrow[rr, "Kir_{\mathfrak{E}}"]&& \Hdot^{\bullet-m}_G(B^{ss}, \mathfrak{E}),
\end{tikzcd}
\end{equation*}
where $m=\codim(B, C_2)$, and $\mathfrak{E}=(K_{C_1/G}^{-1/2}|_{B/G}\otimes K_{C_2/G}^{1/2}|_{B/G}\otimes det(N_{B/C_2}))_{\delta}$ . Both $Kir_{\mathfrak{E}}$ and $res$ are surjective by Theorem \ref{thm:twistedkirwan}.
\end{theorem}
\clearpage
\appendix
\section{Twisted Kirwan Surjectivity}
\subsection{Twisted Kirwan map}
Let $G$ be a linear reductive algebraic group, acting on a smooth projective over affine variety $M$ with a $G$-equivariant ample line bundle $L$. Suppose that there is a $G$-equivariant $\CC$-local system $\mathfrak{L}$ on $M$. Then the natural embedding $M^{ss}/G\hookrightarrow M/G$ induces a map $Kir$ on the twisted cohomology below, called \emph{twisted Kirwan map}:
$$Kir_{\mathfrak{L}}: \Hdot^{\bullet}_G(M, \mathfrak{L})\arw \Hdot^{\bullet}_G(M^{ss}, \mathfrak{L}).$$ 
When $\mathfrak{L}$ is trivial, it is the usual Kirwan map \cite{kirwan1984cohomology} if $G$ acts on $M^{ss}$ freely so that the quotient stack $M^{ss}/G$ is homeomorphic to the symplectic quotient of $M$ by the compact form $G_c$ of $G$. The aim of this appendix is to give a proof of the following twisted Kirwan surjection. Let $M_G=M\times_{BG} EG$ be the homotopy quotient of $M$ by $G$. Then it is known that $\Hdot^{\bullet}_G(M, \mathfrak{L})\cong \Hdot^{\bullet}(M_G, \mathfrak{L})$.
\begin{theorem}[twisted Kirwan surjection]\label{thm:twistedkirwan}
    Let $(M, G, L)$ be above, then the twisted Kirwan map 
    $$Kir: \Hdot^{\bullet}_G(M, \mathfrak{L})\arw \Hdot^{\bullet}_G(M^{ss}, \mathfrak{L})$$ 
    is surjective, where $\mathfrak{L}$ is any $G$-invariant $\CC$-local system on $M$. 
\end{theorem}
We follow the main arguments in \cite{kirwan1984cohomology}. 
\subsection{HKKN stratification of unstable locus} \label{app:hkkn}The first ingredient is the HKKN stratification of the unstable locus $M^{us}:=M-M^{ss}$, which will be reviewed briefly. The reader can refer to \cite{kirwan1984cohomology} for details. 
Let $\lambda: \GG_m\arw G$ be a 1-parameter subgroup (1-PS), and denote $Z\subset M^{\lambda}$ a connected component of the fixed locus $M^{\lambda}$. Then the \emph{blade} of $Z$ is defined to be the following subvariety of $M$:
$$Y_{Z,\lambda}=\{x\in M| \lim_{t\to 0}\lambda(t).x\in Z\}.$$
Then define $S_{\lambda,Z}:=G\cdot Y_{\lambda, Z}$. 

Let $L_{\lambda}$ be the centralizer of $\lambda$ in $G$, and define $U_{\lambda}=\{g\in G| \lim_{t\to 0}\lambda(t)g\lambda(t)^{-1}=1_G\}$ and $P_{\lambda}=\{g\in G| \lim_{t\to 0}\lambda(t)g\lambda(t)^{-1}\text{ exists}\}$. It is shown in \cite{mumford1994geometric} that $P_{\lambda}$ is parabolic, $U_{\lambda}$ (resp. $L_{\lambda}$) is the unipotent radical (resp. Levi component) of $P_{\lambda}$. It is not hard to see that  $Z_{\lambda}$ is $L_{\lambda}$-invariant, and $Y_{\lambda}$ is $P_{\lambda}$-invariant. 
\begin{theorem}[HKKN stratification \cite{kirwan1984cohomology},\cite{halpern2015derived}]\label{thm: hkkn} Let $(M, G, L)$ be above. There is a finite stratification  
$$M^{us}=\bigsqcup_{\alpha} S_{\alpha},$$ indexed by 1-PS's $\lambda_{\alpha}$ and connected components of their fixed locus $M^{\lambda_{\alpha}}$, satisfying the following
\begin{enumerate}
    \item The canonical projection $\pi_{\alpha}: Y_{\alpha}\arw Z_{\alpha}$ given by $x\mapsto \lim_{t\to 0}\lambda(t).x$ is an affine space bundle;
    \item The canonical map $G\times_{P_{\alpha}}Y_{\alpha}\arw G\cdot Y_{\alpha}:=S_{\alpha}$ is an isomorphism.
\end{enumerate}
\end{theorem}
\subsection{Long exact sequences of equivariant cohomology in local coefficient}\label{app:longexact}
The second ingredient is the long exact sequences of equivariant cohomology in local coefficients for the HKKN stratification. The reader can refer to 
Chapter VI of \cite{whitehead2012elements} for details on cohomology in local coefficients. Let $U_{\beta}=M-\bigsqcup_{\alpha\leq \beta} S_{\alpha}$ and $U_{\beta}'=M-\bigsqcup_{\alpha< \beta} S_{\alpha}$. Then for the pair $(U_{\beta}, U'_{\beta})$, there is a long exact sequence:
$$\cdots\arw\Hdot^{\bullet}_G(U_{\beta}, U_{\beta}';\mathfrak{L})\xrightarrow{f_{\beta}}\Hdot^{\bullet}_G(U_{\beta},\mathfrak{L} )\xrightarrow{g_{\beta}} \Hdot^{\bullet}_G(U_{\beta}',\mathfrak{L})\arw \Hdot^{\bullet+1}_G(U_{\beta}, U_{\beta}';\mathfrak{L})\arw\cdots.$$

Note that 
$$\Hdot^{\bullet}_G(U_{\beta}, U_{\beta}';\mathfrak{L})\cong \Hdot^{\bullet}_G(N_{\beta}, N_{\beta}^0;\mathfrak{L})\cong \Hdot^{\bullet-2r_{\beta}}_{G}(S_{\beta}, \mathfrak{L}),$$ where $N_{\beta}$ is the normal bundle of $S_{\beta}\subset M$, and $r_{\beta}=rank (N_{\beta})$. The first $\cong$ comes from the twisted excision in \cite{whitehead2012elements}. The second is the twisted Thom isomorphism, which holds by a Mayer-Vietoris argument as in the usual untwisted Thom isomorphism. 

It is noted that the twisted Kirwan map $Kir_{\mathfrak{L}}$ is the composition $\prod_{\beta} g_{\beta}$. In order to show $Kir$ is surjective, it is sufficient to show that each $g_{\beta}$ is surjective. By means of each long exact sequence, it is equivalent to show that $f_{\beta}$ is injective, or equivalently, the following
$$\Hdot^{\bullet-2r_{\beta}}_G(S_{\beta}, \mathfrak{L})\arw \Hdot^{\bullet}_G(U_{\beta},\mathfrak{L})$$is injective. 

We consider the following composition
$$ \Hdot^{\bullet-2r_{\beta}}_G(S_{\beta}, \mathfrak{L})\arw \Hdot^{\bullet}_G(U_{\beta},\mathfrak{L})\arw \Hdot^{\bullet}_G(S_{\beta}, \mathfrak{L}),$$
where the second map is induced by the inclusion $S_{\beta}\hookrightarrow U_{\beta}$. It is known that the composition is just cupping product with the equivariant Euler class of $N_{\beta}$. Now it is suffices to show that 
$$\Hdot^{\bullet-2r_{\beta}}_G(S_{\beta},\mathfrak{L})\xrightarrow{\cup e_G(N_{\beta})} \Hdot^{\bullet}_G(S_{\beta},\mathfrak{L})$$
is injective.

\subsection{Atiyah-Bott criterion}
The last ingredient lies in the injectivity criterion in Section 13 of \cite{atiyah1983yang}. 
Let $Y$ be a  connected K\"ahler manifold with a compact Lie group $K$ action, $N$ be a $G$-equivariant complex vector bundle on it, and $\mathfrak{L}$ be a $G$-invariant $\CC$-local system on $Y$. The following is the criterion on injectivity of the cup product with $e_K(N)$ due to Atiyah-Bott \cite{atiyah1983yang}:

\begin{proposition}\label{prop: atiyah-bott}
Suppose that there is a subtorus $T_0\leq K$ which acts trivially on $Y$ and that the representation of $T_0$ on the fibre of $N$ at any point of $Y$ has no nonzero fixed vectors. Then 
$$\Hdot^{\bullet-2\,rank(N)}_K(Y,\mathfrak{L})\xrightarrow{\cup e_K(N)} \Hdot^{\bullet}_K(Y,\mathfrak{L})$$
is injective. 
\end{proposition}
\begin{proof}
Let $T_0\subset T\subset K$ be the maximal torus containing $T_0$. Then there is a fibration
\[
\begin{tikzcd}
K/T \arrow[r, hook] & Y_T \arrow[d, two heads]\\ & Y_K
\end{tikzcd}
\]
and the associated Leray spectral sequence degenerates at the $E_2$ page since $K/T$ has no odd cohomology:
$\Hdot^{\bullet}(Y_T,\mathfrak{L}
)\cong \Hdot^{\bullet}(Y_K,\mathfrak{L}
)\otimes \Hdot^{\bullet}(K/T),$
which implies that
$$ \Hdot^{\bullet}(Y_K,\mathfrak{L}
)\subset \Hdot^{\bullet}(Y_T,\mathfrak{L}
)$$
is a direct summand. Therefore, we may assume that $K=T=T_0\times T_1$, where $T_0$ acts on $Y$ trivially. Then 
$$\Hdot^{\bullet}(Y_T, \mathfrak{L})\cong \Hdot^{\bullet}(BT_0)\otimes \Hdot^{\bullet}(Y_{T_1}, \mathfrak{L}),$$
$$\Hdot^{\bullet}(Y_T, \CC)\cong \Hdot^{\bullet}(BT_0)\otimes \Hdot^{\bullet}(Y_{T_1}, \CC).$$
Since $\Hdot^{\bullet}(BT_0)$ is a polynomial ring, and any $0\neq \alpha_0\otimes1\in \Hdot^{\bullet}(BT_0)\otimes \Hdot^{\bullet}(Y_{T_1}, \CC)$ defines an injective map on $\Hdot^{\bullet}(BT_0)\otimes \Hdot^{\bullet}(Y_{T_1}, \mathfrak{L})$. 
Now observe that $\alpha_0\otimes 1+\text{higher degree terms}$ w.r.t. $Y_{T_1}$ in $\Hdot^{\bullet}(BT_0)\otimes \Hdot^{\bullet}(Y_{T_1}, \CC)$ also defines an injective map. Since $T_0$ acts on Y trivially, and the representation of $N$ is non-trivial, $e_K(N)\in \Hdot^{\bullet}(BT_0)\otimes \Hdot^{\bullet}(Y_{T_1}, \CC)$ is such a form, and the criterion concludes. 

\subsection{Proof of twisted Kirwan surjection} By the argument in \ref{app:longexact}, we reduce to show that $$\Hdot^{\bullet-2r_{\beta}}_G(S_{\beta},\mathfrak{L})\xrightarrow{\cup e_G(N_{\beta})} \Hdot^{\bullet}_G(S_{\beta},\mathfrak{L})$$
is injective. Note that
$$\Hdot^{\bullet}_{G}(S_{\beta}, \mathfrak{L})\cong \Hdot^{\bullet}_{P_{\alpha}}(Y_{\beta}, \mathfrak{L})\cong \Hdot^{\bullet}_{L_{\beta}}(Y_{\beta}, \mathfrak{L})\cong\Hdot^{\bullet}_{L_{\beta}}(Z_{\beta}, \mathfrak{L}).$$
The first $\cong$ comes from Theorem \ref{thm: hkkn}(2), and the second one comes from the $LU$ decomposition of $P_{\beta}$ and the last one comes from Theorem \ref{thm: hkkn}(1). Note also that $L_{\beta}$ is reductive, and contains the 1-PS $\lambda_{\beta}$ which acts on $Z_{\beta}$ trivially. By Hilbert-Mumford's criterion, the $\lambda_{\beta}$-weight of $N_{\beta}|_{Z_{\beta}}$ has to be negative. Now to conclude our statement, it is sufficient to apply Proposition \ref{prop: atiyah-bott} then. 
\end{proof}

\bibliography{references202506}

@article{caldararu2003d,
  title={{D}-branes, {B} fields, and {E}xt groups},
  author={C{\u{a}}ld{\u{a}}raru, Andrei and Katz, Sheldon and Sharpe, Eric},
  journal={Adv. Theor. Math. Phys},
  volume={7},
  number={3},
  pages={381-404},
  year={2003}
}

@article{arinkin2012self,
  title={When is the self-intersection of a subvariety a fibration?},
  author={Arinkin, Dima and C{\u{a}}ld{\u{a}}raru, Andrei},
  journal={Advances in Mathematics},
  volume={231},
  number={2},
  pages={815--842},
  year={2012},
  publisher={Elsevier}
}

@article{khan2021virtual,
  title={Virtual excess intersection theory},
  author={Khan, Adeel A},
  journal={Ann. K-Theory},
  volume={6},
  number={3},
  pages={559--570},
  year={2021}
}

@inproceedings{calaque2019shifted,
  title={Shifted cotangent stacks are shifted symplectic},
  author={Calaque, Damien},
  booktitle={Annales de la Facult{\'e} des sciences de Toulouse: Math{\'e}matiques},
  volume={28},
  number={1},
  pages={67--90},
  year={2019}
}

@article{anel2022shifted,
  title={Shifted symplectic reduction of derived critical loci},
  author={Anel, Mathieu and Calaque, Damien},
  journal={Advances in Theoretical and Mathematical Physics},
  volume={26},
  number={6},
  pages={1543--1583},
  year={2022}
}

@article{vasconcelos1967ideals,
  title={Ideals generated by {R}-sequences},
  author={Vasconcelos, Wolmer V},
  journal={Journal of Algebra},
  volume={6},
  number={3},
  pages={309--316},
  year={1967},
  publisher={Elsevier}
}

@article{kubrak2022hodge,
  title={Hodge-to-de {R}ham degeneration for stacks},
  author={Kubrak, Dmitry and Prikhodko, Artem},
  journal={International Mathematics Research Notices},
  volume={2022},
  number={17},
  pages={12852--12939},
  year={2022},
  publisher={Oxford University Press}
}

@book{kirwan1984cohomology,
  title={Cohomology of quotients in symplectic and algebraic geometry},
  author={Kirwan, Frances Clare},
  volume={31},
  year={1984},
  publisher={Princeton university press}
}

@book{mumford1994geometric,
  title={Geometric invariant theory},
  author={Mumford, David and Fogarty, John and Kirwan, Frances},
  volume={34},
  year={1994},
  publisher={Springer Science \& Business Media}
}

@article{pantev2013shifted,
  title={Shifted symplectic structures},
  author={Pantev, Tony and To{\"e}n, Bertrand and Vaqui{\'e}, Michel and Vezzosi, Gabriele},
  journal={Publications math{\'e}matiques de l'IH{\'E}S},
  volume={117},
  pages={271--328},
  year={2013}
}

@article{swan1996hochschild,
  title={Hochschild cohomology of quasiprojective schemes},
  author={Swan, Richard G},
  journal={Journal of Pure and Applied Algebra},
  volume={110},
  number={1},
  pages={57--80},
  year={1996},
  publisher={Elsevier}
}

@article{arinkin2019formality,
  title={Formality of derived intersections and the orbifold {HKR} isomorphism},
  author={Arinkin, Dima and C{\u{a}}ld{\u{a}}raru, Andrei and Hablicsek, M{\'a}rton},
  journal={Journal of Algebra},
  volume={540},
  pages={100--120},
  year={2019},
  publisher={Elsevier}
}

@article{halpern2015derived,
  title={The derived category of a {GIT} quotient},
  author={Halpern-Leistner, Daniel},
  journal={Journal of the American Mathematical Society},
  volume={28},
  number={3},
  pages={871--912},
  year={2015}
}

@article{kapustin2009three,
  title={Three-dimensional topological field theory and symplectic algebraic geometry {I}},
  author={Kapustin, Anton and Rozansky, Lev and Saulina, Natalia},
  journal={Nuclear Physics B},
  volume={816},
  number={3},
  pages={295--355},
  year={2009},
  publisher={Elsevier}
}

@article{toen2014derived,
  title={Derived algebraic geometry},
  author={To{\"e}n, Bertrand},
  journal={EMS Surveys in Mathematical Sciences},
  volume={1},
  number={2},
  pages={153--240},
  year={2014}
}

@phdthesis{lurie2004derived,
  title={Derived algebraic geometry},
  author={Lurie, Jacob},
  year={2004},
  school={Massachusetts Institute of Technology}
}

@article{ben2012loop,
  title={Loop spaces and connections},
  author={Ben-Zvi, David and Nadler, David},
  journal={Journal of Topology},
  volume={5},
  number={2},
  pages={377--430},
  year={2012},
  publisher={Oxford University Press}
}

@article{gunningham2023deformation,
  title={Deformation quantization and perverse sheaves},
  author={Gunningham, Sam and Safronov, Pavel},
  journal={arXiv preprint arXiv:2312.07595},
  year={2023}
}

@article{mladenov2024formality,
  title={Formality of differential graded algebras and complex {L}agrangian submanifolds},
  author={Mladenov, Borislav},
  journal={Selecta Mathematica},
  volume={30},
  number={1},
  pages={8},
  year={2024},
  publisher={Springer}
}

@book{whitehead2012elements,
  title={Elements of homotopy theory},
  author={Whitehead, George W},
  volume={61},
  year={2012},
  publisher={Springer Science \& Business Media}
}

@article{atiyah1983yang,
  title={The {Y}ang-{M}ills equations over {R}iemann surfaces},
  author={Atiyah, Michael Francis and Bott, Raoul},
  journal={Philosophical Transactions of the Royal Society of London. Series A, Mathematical and Physical Sciences},
  volume={308},
  number={1505},
  pages={523--615},
  year={1983},
  publisher={The Royal Society London}
}

@book{barth2003compact,
  title={Compact complex surfaces},
  author={Barth, Wolf and Hulek, Klaus and Peters, Chris and Van de Ven, Antonius},
  volume={4},
  year={2003},
  publisher={Springer Science \& Business Media}
}

@article{marsden1974reduction,
  title={Reduction of symplectic manifolds with symmetry},
  author={Marsden, Jerrold and Weinstein, Alan},
  journal={Reports on mathematical physics},
  volume={5},
  number={1},
  pages={121--130},
  year={1974},
  publisher={Elsevier}
}

@article{freed1999anomalies,
  title={Anomalies in string theory with {D}-branes},
  author={Freed, Daniel S and Witten, Edward},
  journal={Asian Journal of Mathematics},
  volume={3},
  number={4},
  pages={819--852},
  year={1999},
  publisher={International Press of Boston}
}

@incollection{park2025introduction,
  title={An Introduction to Shifted Symplectic Structures},
  author={Park, Hyeonjun and You, Jemin},
  booktitle={Moduli Spaces, Virtual Invariants and Shifted Symplectic Structures},
  pages={37--64},
  year={2025},
  publisher={Springer}
}
\bibliographystyle{plain}

\end{document}